\documentclass[12pt,a4paper]{article}

\usepackage{amsmath}
\usepackage{amsthm}
\usepackage{amssymb}
\usepackage[mathscr]{eucal}
\usepackage{amscd}

\usepackage{enumerate}
\usepackage{ascmac}
\usepackage{subcaption}
\usepackage{float}
\usepackage{cite}
\usepackage{comment}

\usepackage{xcolor}
\usepackage{graphicx}

\newtheorem{theorem}{Theorem}[section]

\theoremstyle{definition}

\numberwithin{equation}{section}

\title{
On the effects of protection zone and directed population flux  
in prey-predator dynamics
}
\author{Kousuke Kuto\footnote{
Corresponding author.
}
\footnote{The first author was supported by
JSPS KAKENHI
Grant-in-Aid for Scientific Research (B)
Grant Number 25K00917. 
}
\\ {\small Department of Applied Mathematics}\\
{\small Waseda University}\\ 
{\small
3-4-1 Ohkubo, Shinjuku-ku, Tokyo
169-8555,
Japan}\\ \\
Kazuhiro Oeda\footnote{The second author was supported by
JSPS KAKENHI
Grant-in-Aid for Scientific Research (C)
Grant Number 23K03241.
\newline
E-mail addresses:
kuto@waseda.jp (K.\,Kuto),\ \ 
kazuoeda@ip.kyusan-u.ac.jp (K.\,Oeda)
}
\\ {\small Center for Fundamental Education}\\
{\small Kyushu Sangyo University}\\ 
{\small
2-3-1 Matsukadai, Higashi-ku, Fukuoka
813-8503,
Japan}
}

\theoremstyle{plain}
\newtheorem{thm}{Theorem}[section]
\newtheorem{prop}[thm]{Proposition}
\newtheorem{lem}[thm]{Lemma}
\newtheorem{cor}[thm]{Corollary}
\theoremstyle{definition}

\newtheorem{rem}{Remark}[section]

\newcommand{\1}{\mathrm{1}\hspace{-0.25em}\mathrm{l}}

\makeatletter

\@addtoreset{equation}{section}
\makeatother

\date{}
\begin{document}
\maketitle 
\begin{abstract}
\noindent
We study a spatial predator–prey model in which prey can enter a protection zone (refuge) inaccessible to predators, while predators exhibit directed movement toward prey-rich regions. The directed movement is modeled by a far-sighted population flux motivated by classical movement rules, in contrast to the more commonly analyzed near-sighted chemotaxis-type mechanisms.
We first establish local-in-time well-posedness for the corresponding nonstationary problem under Neumann boundary conditions, despite the discontinuity induced by the refuge interface. We then investigate the stationary problem, focusing on how the coexistence states emerge and organize globally in parameter space. In particular, we identify the bifurcation threshold for positive steady states from semitrivial predator-only equilibria, and describe the global continuation of the resulting branches. Our analysis reveals that strong directed movement can induce turning-point structures and multiplicity of coexistence steady states, highlighting a nontrivial interplay between spatial protection and predator movement behavior.

\medskip

\end{abstract}

\section{Introduction} 
\hspace{4.5mm} 
In ecological systems, {\it protection zones}--regions into which prey may enter but predators cannot--have long been recognized as an important factor shaping the spatio-temporal dynamics of interacting species.
In mathematical biology, extensive research has been devoted to understanding how protection zones influence population persistence and coexistence.

From the viewpoint of reaction--diffusion models with protection zones (or refuges),
it is important to note that the incorporation of protected regions into spatial
population dynamics predates the later prey--predator literature and goes back at least
to the pioneering works of L{\'o}pez-G{\'o}mez \cite{LG95} and
L{\'o}pez-G{\'o}mez and Sabina de Lis \cite{LGSab}.
In these early contributions, spatially heterogeneous structures with protected patches
were introduced within reaction--diffusion frameworks in a systematic way, and their
effects on persistence/coexistence and long--term dynamics were already analyzed.
This foundational line of research was subsequently developed, unified, and placed in a
broader conceptual setting in the monograph of L{\'o}pez-G{\'o}mez \cite{LG15}, which
revisits the evolution of the protection--zone theory for parabolic equations and systems.
Later, Du and Shi \cite{DS} proposed a specific prey--predator model in which only the prey
are allowed to enter a protection zone, while both species undergo random diffusion.
Within this particular setting, they analyzed the global structure of coexistence steady states and, 
in line with the earlier insights developed in \cite{LG95,LGSab,LG15}, 
showed that enlarging the protection zone can increase the likelihood of prey survival.

Since then, extensive studies have been conducted on Lotka--Volterra-type systems with
protection zones/refuges, providing valuable insights into the structure and
bifurcation behavior of steady states. In particular, various functional responses,
diffusion, and cross-diffusion mechanisms have been incorporated; see, for instance,
\cite{HeZheng2015,HeZheng2017,Oeda2011,Oeda2012,Oeda2017,LiWu2017,LiWuLiu2017,LiYamada2018,LiWu2022}
and references therein.
More recently, diffusion--advection effects have been explored by Ma and Tang
\cite{MaTang2023} and Tang and Chen \cite{TangChen2024}, revealing new qualitative
features absent in purely diffusive models. Protection zones have also been studied in
connection with long-time dynamics, free-boundary problems, and additional ecological
mechanisms, such as strong Allee effects and fear effects; see, for instance,
\cite{JinPengWang2023,SunLei2023,XuZou2024,WangFan2023}. We also refer to the monograph
\cite{LG15} for a systematic account of metasolutions and the role of protected zones
in the dynamics of spatially heterogeneous parabolic equations and systems.
However, much of the existing literature has focused primarily on stationary problems
and arguments based on comparison theorems. Consequently, relatively few studies have
addressed the existence and behavior of time-dependent solutions in cases involving
nonlinear diffusion where comparison theorems are inapplicable, despite their
critical role in ecological dynamics.

\begin{figure}[htbp]
  \centering
  \includegraphics[scale=.3]{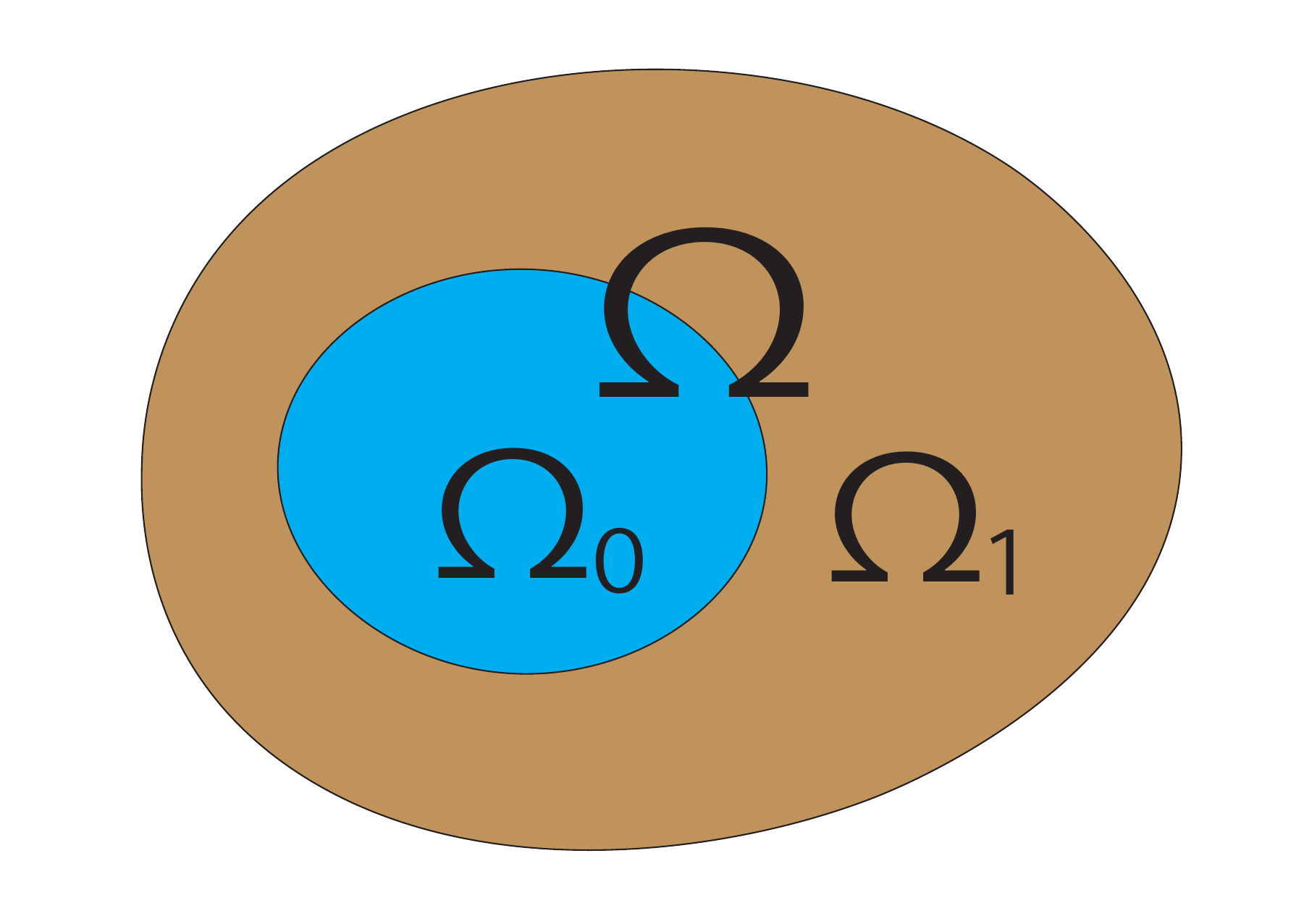}
  \caption{Protection zone $\Omega_{0}$}
  \label{fig1}
\end{figure}

At the same time, directed movement of predators toward prey-rich regions is another important ecological mechanism.
In classical ecological theory, certain far-sighted movement rules, in which movement probabilities depend directly on prey density at a potential destination, are often presented as prototypical descriptions of animal movement (see, e.g., Okubo and Levin \cite{OL}).
However, despite their prominence in the ecological literature, these far-sighted movement rules have received remarkably little mathematical attention within reaction-diffusion frameworks, especially when compared with the extensively studied near-sighted chemotaxis-type mechanisms, where movement responds to prey density at spatial midpoints.
How such far-sighted movement interacts with a protection zone (refuge) therefore remains largely unclear, and the present paper aims to clarify this interaction.

We now introduce the prey-predator model with a protection zone and a
directed movement flux that will be analyzed in this paper.
Let $\Omega\subset\mathbb{R}^{N}$ be a bounded domain (a bounded, connected
open set) whose boundary $\partial\Omega$ is a smooth $(N-1)$-dimensional
manifold.  
The domain $\Omega$ represents the habitat available to the prey population.
From a biological viewpoint the spatial dimension is typically $N=2$ or $N=3$;
however, for the mathematical analysis carried out in this paper,
we assume $N=2$ for the time-dependent problem, whereas for the steady-state 
problem we allow any $N\ge 2$.
A subdomain $\Omega_{0}\subset\Omega$ is designated as a protection zone for prey, which means that the prey can freely move into and out of $\Omega_{0}$,
while predators are excluded from $\Omega_{0}$.
Throughout this paper, 
we assume that $\Omega_{0}$ is a bounded domain with
$\overline{\Omega}_{0}\subset\Omega$, and we define
\[
  \Omega_{1}:=\Omega\setminus\overline{\Omega}_{0},
\]
see Figure \ref{fig1}.
As an additional modeling assumption, we suppose that $\Omega_{1}$ is
connected; in particular, predators are confined to the connected region
$\Omega_{1}$. In this setting, the boundary 
$\partial\Omega_{1}$ consists of two disjoint 
$(N\!-\!1)$-dimensional manifolds;
$\partial\Omega$ and $\partial\Omega_{0}$.

The core mathematical model studied in this paper is the following prey-predator system:
\begin{equation}\label{P}
\begin{cases}
u_{t}=\Delta u+u(\lambda -u-b\1_{\Omega_1}v),\ \
&(x,t)\in\Omega\times (0,T),\vspace{1mm}\\
v_{t}=\Delta v+\alpha \nabla\cdot\biggl[
u^{2}\,\nabla\biggl(\dfrac{v}{u}\biggr)\biggr]
+v(\mu +cu-v),\ \
&(x,t)\in \Omega_{1}\times (0,T),\\
\partial_n u=0,\ \ \ 
&(x,t)\in\partial\Omega\times (0,T),\\
\partial_n v=0,\ \ \ 
&(x,t)\in\partial\Omega_{1}\times (0,T),\\
u(x,0)=u_{0}(x)\ge 0,\ \
&x\in\Omega,\\ 
v(x,0)=v_{0}(x)\ge 0,\ \
&x\in\Omega_{1}.
\end{cases}
\end{equation}  
The unknown functions 
$u(x,t)$ and $v(x,t)$ represent the population densities 
of prey and predator, respectively, 
at location 
$x$ and time $t$.
In the reaction terms of Lotka-Volterra type, 
we assume that $\lambda$ is the intrinsic growth rate of prey 
and is a positive constant. On the other hand, 
the growth rate $\mu$ of predator is assumed to be a real constant 
that may take nonpositive values. 
The term $b\1_{\Omega_1}uv$ represents the decrease in prey due to predation by predator. However, since predators cannot enter 
the protection zone $\Omega_0$, 
we assume that
$b\1_{\Omega_1}$ is a well-type function that 
vanishes in $\overline{\Omega}_0$, where  $b$ is a positive constant  
and
\begin{equation}
\1_{\Omega_1}:=
\begin{cases}
1\quad&\mbox{if}\ x\in\Omega_{1},\\
0\quad&\mbox{if}\ x\in\overline{\Omega}_{0}.
\end{cases}
\end{equation}
On the other hand, the term $cuv$ represents the increase 
in predators due to predation of prey. 
Since the second equation is defined only in the region 
$\Omega_{1}$, 
where predators can reside, 
the assumption that $c$ is a positive constant 
provides a reasonable interpretation. 
That is, the unknown function $u$ is spatially defined in $\Omega$, 
while the second equation itself is defined only in 
$\Omega_{1}$, 
so the function $u$ appearing in the second equation 
is interpreted as being restricted to 
$\Omega_{1}$.
Correspondingly, the function $v$ appearing in the first equation is interpreted as being zero in $\Omega_{0}$ if necessary.
In the Neumann boundary condition for \( u \), given by
$
\partial_{n} u := n \cdot \nabla u = 0$,
the vector \( n \) denotes the outward unit normal vector to \( \partial\Omega \), with \( \Omega \) regarded as the interior. In contrast, in the Neumann boundary condition for \( v \), given by
$
\partial_{n} v := n \cdot \nabla v = 0$ on $\partial\Omega_{1}$,
the vector \( n \) on $\partial\Omega_{1}\cap\partial\Omega$ also denotes the outward unit normal vector, with respect to \( \Omega_{1} \) considered as the interior, and thus coincides with that used in the boundary condition for \( u \). Slight care must be taken regarding the direction of \( n \) on $\partial\Omega_{1}\cap\partial\Omega_{0}$, where \( n \) is defined as the unit normal vector pointing outward from \( \Omega_1 \) into \( \Omega_{0} \), with 
respect to $\Omega_{1}$ considered as the interior.


The linear diffusion terms $\Delta u$ and $\Delta v$ 
represent random diffusion of prey and predator, 
respectively, at comparable rates. It should be noted that we set the random diffusion coefficients to unity for both species to avoid notational complexity; however, the results presented in this paper remain verifiable for general positive coefficients.
The nonlinear diffusion term $\nabla\cdot[u^{2}\nabla (v/u)]$
in the second equation 
models the tendency of predator to move toward regions 
with higher prey density, 
and the coefficient $\alpha$ is 
assumed to be a nonnegative constant.
The chemotaxis term $-\nabla\cdot(v\nabla u)$ is well-known as 
a nonlinear term that describes the pursuit behavior 
of predators tracking prey. However, the term 
$\nabla\cdot [u^2\nabla (v/u)]$ 
is also a typical modeling approach from 
the perspective of directed population fluxes. 
According to \cite{OL}, 
which is frequently cited in the context of modeling 
diffusion processes in biological species, 
the nonlinear diffusion term $\nabla\cdot [u^2\nabla (v/u)]$ 
naturally arises when considering a discrete patch model 
in which predators at each site move to neighboring sites 
during short time intervals, with a higher probability of moving toward 
sites where prey density is higher. 
Through a continuum limit, this leads to the emergence of 
the nonlinear diffusion term $\nabla\cdot [u^2\nabla (v/u)]$. 
In other words, this diffusion form naturally appears 
when individuals move based on the prey population density 
at the destination site.
On the other hand, if movement is assumed to be based on the prey density at the midpoint between the origin and destination sites, then the chemotaxis 
term $-\nabla\cdot(v\nabla u)$ arises. 
In this sense, both nonlinear diffusion forms 
are considered prototypical from the viewpoint of biological 
diffusion processes, and indeed, 
they are presented in parallel in \cite{OL}.
However, compared to the extensive studies of various biological 
models involving chemotaxis terms, the nonlinear diffusion term
$\nabla\cdot [u^2\nabla (v/u)]$ involving 
population flux by attractive transition
has been less explored from the perspective of 
partial differential equations.

The purpose of this paper is to establish the local-in-time 
solvability of \eqref{P}, 
and to further study the existence and 
global bifurcation structure of coexistence steady states.

In constructing nonstationary solutions, 
the analysis becomes significantly more difficult than 
in conventional semilinear prey-predator systems due to 
the presense of the protection zone $\Omega_{0}$ and
the nonlinear diffusion term $\nabla\cdot[u^{2}\nabla (v/u)]$. 
In the two dimensional case where $N=2$, 
this paper will apply the theory of nonlinear evolution equations, 
such as that developed by Yagi \cite{Ya}, 
to construct local-in-time solutions in appropriate 
function spaces. 

For the stationary problem, we first identify and plot the pairs of parameters
\((\lambda,\mu)\) for which coexistence steady states bifurcate from semi-trivial
solutions, where one of the species (prey or predator) is extinct.
We then investigate the qualitative behavior of the branches of coexistence
steady states bifurcating from such semi-trivial solutions (see Figure~2).
Furthermore, we clarify the global structure of these bifurcation branches by
analyzing their asymptotic behavior in the limit where the directed population
flux $\alpha$ becomes infinitely strong.
This is achieved by deriving and studying the global bifurcation structure of
the solution set of the corresponding limiting system.
In particular, we show that the combined effects of the protection zone
\(\Omega_{0}\) and the directed population flux can reverse the direction of
bifurcation and lead to the emergence of coexistence branches with turning
points (saddle--node bifurcations), see Figure 4.
From a biological perspective, this result suggests that increasing the strength
of directed predator movement, which might naively be expected to benefit the
predator, can instead allow the prey to persist at lower intrinsic growth
rates, at least from the viewpoint of stationary solutions.

The structure of this paper is as follows. 
In Section 2, we investigate the existence of time-dependent 
solutions to \eqref{P}. 
In Section 3, we present results on the bifurcation structure 
of steady-state solutions and their asymptotic behavior as 
$\alpha$ becomes large. 
In Section~4, we investigate the bifurcation of coexistence steady states
from semi-trivial solutions and analyze the behavior of the resulting
global bifurcation branches.
Section~5 is devoted to the asymptotic behavior of coexistence steady states
$(u(x,\alpha), v(x,\alpha))$ as the directed movement strength $\alpha$
tends to infinity.
More precisely, we show that as $\alpha\to\infty$, along suitable subsequences,
coexistence steady states exhibit one of two possible limiting behaviors:
either they converge to a prey-only state, or the rescaled solutions
$(\alpha u(x,\alpha),\, v(x,\alpha))$ converge to a positive function pair
$(w(x), v(x))$.
In the latter case, we derive a limiting nonlinear elliptic system
satisfied by $(w,v)$.
In Section~6, we study the global bifurcation structure of positive solutions
$(w,v)$ to this limiting system.
In particular, near the blow-up point $\lambda=0$ of the $w$-component,
we provide a detailed description of the bifurcation branch,
taking $\lambda$ as the bifurcation parameter.

\section{Local-in-time solvability}
Throughout this section, devoted to the existence of nonstationary 
solutions, we assume that the spatial dimension is two $(N=2)$ 
unless otherwise stated.
We begin by observing that, in \eqref{P},
 nonlinear diffusion can alternatively be expressed in the following form. 
\begin{equation}\label{exps}
\nabla\cdot\biggl[
u^{2}\,\nabla\biggl(\dfrac{v}{u}\biggr)\biggr]
=\nabla\cdot (u\nabla v-v\nabla u)
=u\Delta v-v\Delta u.
\end{equation}
Throughout this paper, we shall employ these expressions interchangeably, depending on the context of the argument.

When discussing the existence of solutions to \eqref{P} 
from the perspective of quasilinear abstract evolution equations, 
the central expression of \eqref{exps} is relatively 
more suitable, and \eqref{P} can be reformulated as follows.
\begin{equation}\label{P2}
\begin{cases}
u_{t}=\Delta u+u(\lambda -u-b\1_{\Omega_1}v),\ 
&(x,t)\in\Omega\times (0,T),\\
v_{t}=\nabla\cdot \{\,(1+\alpha u)\nabla v\,\}-
\alpha\nabla\cdot ( v\nabla  u)
+v(\mu +cu-v),\ 
&(x,t)\in \Omega_{1}\times (0,T),\\
\partial_n u=0,\ 
&(x,t)\in\partial\Omega\times (0,T),\\
\partial_n v=0,\ 
&(x,t)\in\partial\Omega_{1} \times (0,T),\\
u(x,0)=u_{0}(x)\ge 0,\ 
&x\in\Omega,\\ 
v(x,0)=v_{0}(x)\ge 0,\ 
&x\in\Omega_{1}.
\end{cases}
\end{equation}  
There are relatively few studies on time-dependent problems for Lotka–Volterra systems with nonlinear terms of the form 
\eqref{exps}.
In the case where the protected region $\Omega_{0}$ is empty and $\Omega = \Omega_{1}$ in \eqref{P2}, 
Haihoff and Yokota \cite{HY22} established the existence and uniqueness of global-in-time solutions and their asymptotic behavior as $t \to \infty$ for the one or two-dimensional case where $N\le 2$.
Furthermore, they succeeded in constructing global-in-time weak solutions even in the case where $N=3$
Regarding a weak symbiotic systems with a couple of the same nonlinear diffusion terms, 
Kato and Kuto \cite{KK25} have obtained the existence and uniqueness of global-in-time solutions 
in case where $N\le 3$
as well as the asymptotic behavior of solutions as $t \to \infty$.
However, in the case of \eqref{P2} with a protection zone $\Omega_{0}$, the regularity of solutions and a priori estimates differ fundamentally from those in the case without a protected region. Consequently, the discussion of solvability becomes considerably more difficult.

In this subsection, we regard \eqref{P2} 
as a quasilinear evolution equation in a suitable function space, 
and construct a local-in-time solution based on the theory 
of quasilinear evolution equations developed, for instance, 
by Yagi \cite{Ya}.
In rewriting \eqref{P2} as a quasilinear evolution equation, 
we employ 
$$\mathbb{L}^{2}:=L^{2}(\Omega )\times 
L^{2}(\Omega_{1})
$$
as the underlying functional space.
In the following discussion, $\Delta_{\Omega}$ denotes the Laplace
operator with domain
\[
H^{2}_{n}(\Omega) := \left\{\, u \in H^{2}(\Omega) \;:\; \partial_{n} u = 0 \ \ \text{on} \ \partial\Omega \,\right\}.
\]
Hence,
$\Delta_{\Omega}$ with domain $D(\Delta_{\Omega})=H^{2}_{n}(\Omega )$
is a densely defined closed operator in $L^{2}(\Omega)$.
Similarly, $\Delta_{\Omega_{1}}$ is defined by
the Laplace operator with domain
\[
H^{2}_{n}(\Omega_{1}) 
:= \left\{\, u \in H^{2}(\Omega_{1}) \;:\; \partial_{n} u = 0 \ \ \text{on} \ \partial\Omega_{1} \,\right\}.
\]
Here we assume that the initial data 
$\boldsymbol{U}_{0}:=(u_{0}, v_{0})$ satisfies
$u_{0}\ge 0$ in $\Omega$ and $v_{0}\ge 0$ in $\Omega_{1}$ and
$$
\boldsymbol{U}_{0}\in H^{1+\varepsilon}(\Omega)\times 
H^{1+\varepsilon}(\Omega_{1})=:\mathbb{H}^{1+\varepsilon}$$
for some $\varepsilon\in (0,1/2)$.
For each $\boldsymbol{U}\in\mathbb{H}^{1+\varepsilon_{1}}$,
where $0<\varepsilon_{1}<\varepsilon$, 
we define the linear operator $A(\boldsymbol{U})$ of $\mathbb{L}^{2}$
with domain
$$
D(A(\boldsymbol{U}))=
H^{2}_{n}(\Omega )\times H^{2}_{n}(\Omega_{1})=:\mathbb{H}^{2}_{n}$$
by
\begin{equation}\label{Adef}
A(\boldsymbol{U}):=
\biggl[
\begin{array}{ll} 
A_{11} & 0\\
A_{21}(v) & A_{22}(u)
\end{array}
\biggr],
\end{equation}
where
$$
\begin{cases}
A_{11}\phi:=-\Delta \phi+\phi\ &\mbox{with domain}\ 
D(A_{11})=H^{2}_{n}(\Omega ), \\
A_{21}(v)\phi:=\alpha\nabla\cdot (v\nabla\phi)\ &\mbox{with domain}\  
D(A_{21}(v))=H^{2}_{n}(\Omega ), \\
A_{22}(u)\psi:=-\nabla\cdot\big((1+\alpha\chi(\operatorname{Re} u))\nabla\psi\big)+\psi &\mbox{with domain}\ 
D(A_{22}(u))=H^{2}_{n}(\Omega_{1})
\end{cases}
$$
for $\boldsymbol{U}=(u,v)
\in\mathbb{H}^{1+\varepsilon_{1}}
$.
Here $\chi (s)$ is the cut function satisfying
$\chi (s)=s$ if $s\ge 0$,
$\chi (s)=-\delta $ if $s\le -\delta$ and
$\chi (s)$ is non-decreasing and smooth for any $s\in\mathbb{R}$
with some $\delta\in (0,1/(2\alpha))$.
It is noted that the solvability of \eqref{P2} with $\alpha=0$
can be carried out based on the semilinear parabolic theory.
Furthermore we introduce
the nonlinear map
$F(\boldsymbol{U} )\in\mathbb{L}^{2}$ by
\begin{equation}\label{Fdef}
F(\boldsymbol{U})=\biggl[
\begin{array}{l}
u(\lambda +1-u-b\1_{\Omega_1}v)\\
v(\mu +1+cu-v)
\end{array}
\biggr]
\end{equation}
for $\boldsymbol{U}=(u,v)\in\mathbb{H}^{1+\varepsilon_{1}}$.
Consequently, 
\eqref{P2} reduces to the following quasilinear abstract
evolution equation
in $\mathbb{L}^{2}$:
\begin{equation}\label{evo}
\begin{cases}
\dfrac{d\,\boldsymbol{U}}{dt}+A(\boldsymbol{U})\boldsymbol{U}=
F(\boldsymbol{U}),\qquad t\in (0,T),\\
\boldsymbol{U}(0)=\boldsymbol{U}_{0}.
\end{cases}
\end{equation}
In the definition of $A_{22}(u)$ we have introduced the cut-off
\[
A_{22}(u)\psi=-\nabla\cdot\big((1+\alpha\chi(\operatorname{Re} u))\nabla\psi\big)+\psi
\]
rather than simply using $1+\alpha u$. The reason is that in the abstract framework of
Yagi~\cite{Ya} 
the evolution equation is studied in the complexified Banach spaces,
so that the unknown $u$ has to be regarded as complex-valued. If one directly inserted $1+\alpha u$,
the leading coefficient could take complex values and 
the ellipticity of the operator would be lost.
By taking $\operatorname{Re} u$ and applying the cut-off $\chi$, 
the coefficient is always real and
uniformly bounded from below by a positive constant, 
which guarantees the normal ellipticity and
sectoriality required in (A1)-(A2) below. In the biologically relevant situation $u\ge0$, we simply have
$\chi(\operatorname{Re} u)=u$, hence the cut-off does not change the original system.
In addition,
it is worth noting that only $A_{22}(u)$ requires this modification.
For $A_{11}$ the coefficients are constant, so ellipticity is never lost even for complex-valued
functions. The block $A_{21}(v)$ is an off-diagonal term of lower order, 
and thus its coefficients
may take complex values without affecting ellipticity. 
In contrast, the principal part of $A_{22}(u)$
depends directly on $u$ and would lose ellipticity if $u$ were allowed to be complex.
Therefore, the cut-off with $\chi(\operatorname{Re}u)$ 
is introduced exclusively in $A_{22}$.

In order to construct a local-in-time solution of \eqref{evo}, 
we check the applicability conditions of the
existence theorem by Yagi
\cite[Theorem~5.5]{Ya}.  
As structural assumptions for $A$, we list the conditions given in 
\cite[p.~202]{Ya}:
\begin{itemize}
  \item[(A1)] there exists $\omega\in (0,\pi/2)$ such that
the spectrum of $A(\boldsymbol{U})$ is contained in
$\Sigma_{\omega} := \{\, z \in \mathbb{C} \,;\, |\arg z| < \omega \,\}$ for any $\boldsymbol{U} \in K_{R}:=\{\,\boldsymbol{U}\in 
\mathbb{H}^{1+\varepsilon_{1}}
\,:\,
\|\boldsymbol{U}\|_{\mathbb{H}^{1+\varepsilon_{1}}}<R\,\}$;  
  \hspace{1.5em} 
  \item[(A2)] there exists $M_{R} \ge 1$ such that
  \[
    \|(z - A(\boldsymbol{U}))^{-1}\|_{\mathcal{L}(\mathbb{L}^2)} \le \frac{M_{R}}{|z|} \quad \text{for all } z \notin \Sigma_{\omega},\ 
\boldsymbol{U} \in K_{R};
  \]
  \item[(A3)] there exists an exponent $\nu\in (0,1]$ such that
$D(A(\boldsymbol{U}_{1}))\subset D(A(\boldsymbol{U}_{2})^{\nu})$ for any pair 
$\boldsymbol{U}_{1}$, $\boldsymbol{U}_{2}\in K_{R}$;
\item[(A4)] there exists $C=C_{R}>0$ such that
$$\|A(\boldsymbol{U}_{1})^{\nu}
[A(\boldsymbol{U}_{1})^{-1}-
A(\boldsymbol{U}_{2})^{-1}]\|_{\mathcal{L}(\mathbb{L}^{2})}\le C_{R}
\|\boldsymbol{U}_{1}-\boldsymbol{U}_{2}\|_{\mathcal{Y}}$$ 
for all $\boldsymbol{U}_{1}$, 
$\boldsymbol{U}_{2}\in K_{R}$ with some Banach space
$\mathbb{H}^{1+\varepsilon_{1}}\subset\mathcal{Y}\subset\mathbb{L}^{2}$ 
with continuous embeddings;
\item[(A5)] there are two exponents $0\le \beta_{0}<\beta_{1}<1$ such that,
for any $\boldsymbol{U}\in K_{R}$, $D(A(\boldsymbol{U})^{\beta_{0} })\subset \mathcal{Y}$ 
and
$D(A(\boldsymbol{U})^{\beta_{1} })\subset\mathbb{H}^{1+\varepsilon_{1} }$ with 
the estimates
$$
\begin{cases}
\|\boldsymbol{\varPhi}\|_{\mathcal{Y}}\le C_{1}\|A(\boldsymbol{U})^{\beta_{0}}\boldsymbol{\varPhi}\|_{\mathbb{L}^{2}}\quad&\mbox{for all}\ \ \boldsymbol{\varPhi}\in D(A(\boldsymbol{U})^{\beta_{0} }),\ \ 
\boldsymbol{U}\in K_{R},\\
\|\boldsymbol{\varPhi}\|_{\mathbb{H}^{1+\varepsilon_{1}}}\le C_{2}
\|A(\boldsymbol{U})^{\beta_{1}}\boldsymbol{\varPhi}\|_{\mathbb{L}^{2}}\quad
&\mbox{for all}\ \ \boldsymbol{\varPhi}\in D(A(\boldsymbol{U})^{\beta_{1} }),\ \ 
\boldsymbol{U}\in K_{R},
\end{cases}
$$
where $C_{i}>0$
$(i=1,2)$ being some constants;
\item[(A6)]
the exponents satisfy the relations
$$
0\le\beta_{0}<\beta_{1}<\nu\le 1\quad\mbox{and}\quad
1+\beta_{0} <\beta_{1}+\nu .
$$
\end{itemize}
The conditions (A1)-(A6) are the structural hypotheses in Yagi's abstract theory
ensuring the applicability of the analytic semigroup approach to quasilinear
parabolic problems. In particular, (A1) and (A2) guarantee that 
$-A(\boldsymbol{U})$
generates an analytic semigroup uniformly in $\boldsymbol{U}$, 
which is the basis for the time-local well-posedness. 
Condition (A3) provides a nesting property of domains,
allowing the operator family $\{A(\boldsymbol{U})\}$ to be treated consistently as
$\boldsymbol{U}$ varies. 
Condition (A4) requires Lipschitz continuity of $A(\boldsymbol{U})$
with respect to $\boldsymbol{U}$, which is essential for applying the contraction mapping
principle in the nonlinear setting. 
Finally, (A5) and (A6) concern fractional powers
of $A(\boldsymbol{U})$ and their interpolation properties; these connect the initial data
spaces with the time regularity of solutions and are indispensable for obtaining
solutions in the desired function class.

We now verify that these conditions are satisfied for our problem.
By a standard argument, one can check that the operator 
$A(\boldsymbol{U})$ is sectorial and satisfies (A1) and (A2); see, for instance, 
the discussion in \cite[p.~244]{Ya}.
Moreover, (A3) is evidently fulfilled by taking $\nu = 1$.
To verify (A4) with $\nu=1$, note that
\[
A(\boldsymbol{U}_{1})\big[A(\boldsymbol{U}_{1})^{-1}-A(\boldsymbol{U}_{2})^{-1}\big]
= -\big(A(\boldsymbol{U}_{1})-A(\boldsymbol{U}_{2})\big)A(\boldsymbol{U}_{2})^{-1},
\]
which shows that (A4) is equivalent to the estimate
\begin{equation}\label{A4}
\|\,[A(\boldsymbol{U}_{1})
-A(\boldsymbol{U}_{2})]\,\boldsymbol{\varPhi}\,\|_{\mathbb{L}^{2}}\le
\widetilde{C}_{R}\|\boldsymbol{U}_{1}-\boldsymbol{U}_{2}\|_{\mathcal{Y}}
\|\boldsymbol{\varPhi}\|_{\mathbb{H}^{2}}
\end{equation}
for all $\boldsymbol{U}_{1}$, 
$\boldsymbol{U}_{2}\in K_{R}$ and $\boldsymbol{\varPhi}\in\mathbb{H}^{2}$.
As in \cite[Section~8.3]{Ya}, one can verify that  
the operator $A$, defined by \eqref{Adef}, satisfies \eqref{A4}  
for $\mathcal{Y} = \mathbb{H}^{1+\varepsilon_0}$ 
with $0 < \varepsilon_0 < \varepsilon_{1}$. 
Taking into account the correspondence 
between the domain of the fractional power of 
$A(\boldsymbol{U})$ and the interpolation spaces,
\begin{equation}\label{inter}
D(A(\boldsymbol{U})^{\theta})=
\begin{cases}
\mathbb{H}^{2\theta}, & 0<\theta<3/4,\\
\mathbb{H}^{2\theta}_{n}, & 3/4<\theta\le 1,
\end{cases}
\end{equation}
valid for all $\boldsymbol{U}\in\mathbb{H}^{1+\varepsilon_{1}}$
(see \cite[Section~8.2]{Ya}),
we deduce that (A5) holds with 
$\beta_{0} = (1+\varepsilon_{0})/2$ and $\beta_{1} = (1+\varepsilon_{1})/2$.
Consequently, (A6) is also satisfied.

Next, we state the conditions for the nonlinear mapping 
$F$, as extracted from \cite[p.~219]{Ya}:

\begin{itemize}
\item[(F1)] There exist a continuous increasing function
$\varphi(\cdot)$ and a Banach space $\mathcal{W}$
with continuous embedding $\mathcal{W}\subset\mathbb{H}^{1+\varepsilon_{1}}$
such that
\begin{equation}
\begin{split}
&\|F(\boldsymbol{U}_{1})-F(\boldsymbol{U}_{2})\|_{\mathbb{L}^{2}} \\
&\quad \le
\varphi\bigl(\|\boldsymbol{U}_{1}\|_{\mathbb{H}^{1+\varepsilon_{1}}}
 + \|\boldsymbol{U}_{2}\|_{\mathbb{H}^{1+\varepsilon_{1}}}\bigr) \\
&\qquad\times
\Bigl(
   \|\boldsymbol{U}_{1}-\boldsymbol{U}_{2}\|_{\mathcal{W}}
   +(\|\boldsymbol{U}_{1}\|_{\mathcal{W}}
    +\|\boldsymbol{U}_{2}\|_{\mathcal{W}})
    \|\boldsymbol{U}_{1}-\boldsymbol{U}_{2}\|_{\mathbb{H}^{1+\varepsilon_{1}}}
 \Bigr)
\end{split}
\nonumber
\end{equation}
for all $\boldsymbol{U}_{1},\boldsymbol{U}_{2}\in\mathcal{W}$.

\item[(F2)] There exists $\beta_{1}<\eta<1$ such that,
for any $\boldsymbol{U}\in K_{R}$, 
$D(A(\boldsymbol{U})^{\eta})\subset \mathcal{W}$ with the estimate
\[
\|\boldsymbol{\varPhi}\|_{\mathcal{W}}\le C_{3}\|
A(\boldsymbol{U})^{\eta}\boldsymbol{\varPhi}\|_{\mathbb{L}^{2}},
\quad\text{for all }\boldsymbol{\varPhi}\in D(A(\boldsymbol{U})^{\eta}),
\]
where $C_{3}>0$ is a constant.
\end{itemize}

The conditions (F1) and (F2) govern the nonlinear mapping $F(\boldsymbol{U})$.
Assumption (F1) provides a local Lipschitz estimate of $F$, which is indispensable
to guarantee uniqueness of solutions via a fixed point argument.
Assumption (F2) ensures that the auxiliary space $\mathcal{W}$ used in (F1)
is compatible with the fractional domain of $A(\boldsymbol{U})$,
so that Sobolev embeddings and interpolation theory can be applied
to control the nonlinear terms. Together, (F1) and (F2) provide
the necessary framework to handle nonlinearities within the analytic
semigroup setting.

In the case $N=2$, since $\mathbb{H}^{s}$ is continuously embedded into 
$C(\overline{\Omega})\times C(\overline{\Omega}_{1})$ for any $s>1$, 
one can verify that the mapping $\boldsymbol{F}$ defined in \eqref{Fdef} 
satisfies (F1) with $\mathcal{W}=\mathbb{H}^{1+\varepsilon_{2}}$
and $\varphi(\xi)=C(1+\xi)$, where $\varepsilon<\varepsilon_{2}<1/2$
and $C>0$ is a constant.
Condition (F2) also follows, since \eqref{inter} implies
$D(A(\boldsymbol{U})^{\eta})=\mathcal{W}$ with 
$\eta=(1+\varepsilon_{2})/2$ for any $\boldsymbol{U}\in K_{R}$.
In addition, the initial value $\boldsymbol{U}_{0}$ satisfies
\[
\boldsymbol{U}_{0}\in D(A(\boldsymbol{U}_{0})^{\gamma}),
\quad \beta_{1}<\gamma\le 1,
\]
with $\varGamma=(1+\varepsilon)/2$.

Consequently, we may apply \cite[Theorems~5.5 and 5.6]{Ya}
to deduce the unique existence of local-in-time solutions to \eqref{evo}.
\begin{theorem}\label{localthm}
Assume $N=2$.
Under the assumptions \textup{(A1)-(A6)}, \textup{(F1)} and \textup{(F2)},
there exists a unique local-in-time solution $\boldsymbol{U}$ of \eqref{evo}
on some interval $[0,T]$, with
\[
\boldsymbol{U}\in
C([0,T]; \mathbb{H}^{1+\varepsilon_{1}})
\cap C^{\gamma-\beta_{0}}([0,T]; \mathbb{H}^{1+\varepsilon_{0}})
\cap C^{1}((0,T];\mathbb{L}^{2}),
\]
and
\[
A(\boldsymbol{U})^{\gamma}\boldsymbol{U}\in C([0,T]; \mathbb{L}^{2}),
\]
together with
\[
\dfrac{d\boldsymbol{U}}{dt},\
A(\boldsymbol{U})\boldsymbol{U}\in\mathcal{F}^{\gamma,\sigma}
((0,T]; \mathbb{L}^{2}), \quad \text{for any }\sigma\in (0,1).
\]
Here the existence time $T>0$ depends only on
$\|A(\boldsymbol{U}_{0})^{\gamma}\boldsymbol{U}_{0}\|_{\mathbb{L}^{2}}$.
Moreover, the solution $\boldsymbol{U}$ satisfies the estimate
\[
\sup_{0\le t \le T}\|A(\boldsymbol{U}(t))^{\gamma}\boldsymbol{U}(t)\|_{\mathbb{L}^{2}}
+ \Bigl\|\dfrac{d\boldsymbol{U}}{dt}\Bigr\|_{\mathcal{F}^{\gamma, \sigma}}
+ \|A(\boldsymbol{U})\boldsymbol{U}\|_{\mathcal{F}^{\gamma, \sigma}}
\le C
\]
for some $C>0$ depending on
$\|A(\boldsymbol{U}_{0})^{\gamma}\boldsymbol{U}_{0}\|_{\mathbb{L}^{2}}$.
\end{theorem}

In Theorem~\ref{localthm}, 
the norm in the space $\mathcal{F}^{\gamma,\sigma}((0,T]; \mathbb{L}^{2})$
is defined by
\begin{equation}
\|\boldsymbol{V}\|_{\mathcal{F}^{\gamma,\sigma}}
=\sup_{0 < t \leq T} t^{1 - \gamma} \|\boldsymbol{V}(t)\|_{\mathbb{L}^{2}}
+ \sup_{0 < s < t \leq T}
\frac{s^{1 - \gamma + \sigma}\,
\|\boldsymbol{V}(t) - \boldsymbol{V}(s)\|_{\mathbb{L}^{2}}}{(t - s)^{\sigma}}.
\nonumber
\end{equation}
Although the regularity class in Theorem~2.1 is stated in terms of
interpolation spaces and fractional powers of $A(\boldsymbol{U})$,
in our present setting these spaces coincide with the standard Sobolev spaces
introduced above. In particular, since $D(A(\boldsymbol{U}))=H^2_n(\Omega)\times
H^2_n(\Omega_1)$ and the interpolation identities \eqref{inter} hold,
the solution class given in Theorem~2.1 is equivalent to the one stated below.

\begin{cor}
\label{cor:local}
Suppose that $N=2$ and $\varepsilon\in(0,1/2)$.
For any nonnegative
$
\boldsymbol{U}_0=(u_0,v_0)\in \mathbb{H}^{1+\varepsilon},
$
there exists 
$
T=T\big(\|\boldsymbol{U}_0\|_{\mathbb{H}^{1+\varepsilon}}\big)>0
>0$
such that \eqref{P2}
(equivalently \eqref{P})
admits a unique local-in-time solution
\[
\boldsymbol{U}=(u,v)\in C([0,T];\mathbb{L}^2)
\cap C((0,T];\mathbb{H}^2_n),
\]
with $\boldsymbol{U}_t\in C((0,T];\mathbb{L}^2)$.
\end{cor}
It should be noted that the
corresponding solution remains nonnegative for all $t\in[0,T]$.
Moreover, if the initial data are nonnegative and not identically zero,
then the solution becomes strictly positive.
This property follows from the parabolic maximum principle and is
also discussed in \cite[Section~8.4]{Ya}.

In addition, we make the following observation on the regularity of solutions.
Set $a(x,t):=1+\alpha\,\chi(\operatorname{Re}u(x,t))$.
Since $a$ depends only on $u$ and $u(\cdot,t)\ge0$ for $t\in[0,T]$,
we have $a(\cdot,t)\ge 1$ (uniform ellipticity).
Fix any $\tau\in(0,T)$. On $[\tau,T]$, the $u$-equation has constant principal part
and bounded right-hand side $u(\lambda-u-b\1_{\Omega_{1}}v)$ with 
$b\1_{\Omega_{1}}\in L^\infty$;
hence by the $L^p$-parabolic regularity, one obtains
\[
u(\cdot,t)\in W^{2,p}(\Omega)\ \text{for any }p\in (1,\infty ),
\quad t\in[\tau,T],
\]
whence $u(\cdot,t)\in 
C^{1,\theta}(\overline{\Omega})$ for $p>2$ and some $\theta\in(0,1)$.
Consequently,
\[
a(\cdot,t)=1+\alpha\,\chi(\operatorname{Re}u(\cdot,t))\in W^{1,\infty}(\Omega_1)
\quad\text{uniformly for }t\in[\tau,T].
\]
With this coefficient regularity and $a(\cdot,t)\ge 1$, 
the $v$-equation is a parabolic Neumann problem whose principal operator
$-\nabla\cdot(a\nabla\cdot)$ is uniformly elliptic with $a\in W^{1,\infty}$.
The standard parabolic $L^2$-theory then yields
\[
v(\cdot,t)\in H^2_n(\Omega_1)\quad\text{for all }t\in[\tau,T].
\]
Thus, in our setting the components satisfy
\[
u(\cdot,t)\in H^2_n(\Omega),\quad v(\cdot,t)\in H^2_n(\Omega_1)\qquad\text{for every }t>0,
\]
and the coefficient $a(\cdot,t)$ is uniformly elliptic with $a\in W^{1,\infty}(\Omega_1)$
for every interval.
\section{Main results on the stationary problem}
In this section, we present the main results concerning the stationary 
solutions of \eqref{P}.  
By \eqref{exps}, the stationary system takes the form  
\begin{equation}\label{sp}
\begin{cases}
\Delta u + u(\lambda - u - b\1_{\Omega_1}v) = 0, & \text{in } \Omega,\\[3pt]
\Delta v + \alpha (u\Delta v - v\Delta u) + v(\mu + c u - v) = 0, 
& \text{in } \Omega_{1},\\[3pt]
\partial_n u = 0, & \text{on } \partial\Omega,\\
\partial_n v = 0, & \text{on } \partial\Omega_{1}.
\end{cases}
\end{equation}
Since only nonnegative states are biologically relevant, 
\eqref{sp} may be rewritten in the equivalent form
\begin{equation}\label{sp2}
\begin{cases}
\Delta u + u(\lambda - u - b\1_{\Omega_1}v) = 0, & \text{in } \Omega,\\[3pt]
\displaystyle 
\Delta v + \frac{v}{1+\alpha u}
\bigl\{\alpha u(\lambda - u - b v) + \mu + c u - v\bigr\} = 0, 
& \text{in } \Omega_{1},\\[3pt]
\partial_n u = 0, & \text{on } \partial\Omega,\\
\partial_n v = 0, & \text{on } \partial\Omega_{1},
\end{cases}
\end{equation}
which is
obtained by substituting the first equation into the second one in \eqref{sp}.  For a class of solutions of
\eqref{sp} (or \eqref{sp2}), we introduce the Sobolev space
\begin{equation}\label{w2pn}
X=W^{2,p}_n(\Omega)\times W^{2,p}_n(\Omega_1)\quad\mbox{for}\quad p>N,
\end{equation}
where 
$W^{2,p}_n(O)=\{\,w\in W^{2,p}(O):\partial_n w=0 \ \text{on }\partial O\,\}$
$(O=\Omega\ \mbox{or}\ \Omega_{1})$.
We call $(u,v)\in X$ a solution of \eqref{sp} (or equivalently \eqref{sp2}) 
if it satisfies these equations.  
It should be noted that, due to the discontinuity of $b\,\1_{\Omega_{1}}$ 
across the interface $\partial\Omega_{0}$, such a solution $(u,v)\in X$ satisfies 
$v\in C^{2}(\overline{\Omega}_{1})$, whereas $u$ does not necessarily belong to 
$C^{2}(\overline{\Omega})$.
A pair $(u,v)\in X$ is called a \emph{positive solution} of \eqref{sp} (equivalently \eqref{sp2})
if $u>0$ in $\Omega$ and $v>0$ in $\Omega_{1}$. 
In spite of the discontinuity of $b\,\1_{\Omega_{1}}$, it follows from 
the Sobolev embedding theorem and the
Bony strong maximum principle (e.g., \cite{Bo}, \cite[Theorem~7.1]{L}) that 
any nonnegative solution $(u,v)\in X$ whose components are both 
nontrivial must actually be strictly positive.
Such a positive solution corresponds to a biological coexistence of prey and predator 
within the refuge-type domain $\Omega_{1}$.  
The main results involve the global bifurcation structure of positive 
solutions of \eqref{sp} (equivalently \eqref{sp2}), 
together with their asymptotic behavior as the nonlinear diffusion parameter 
$\alpha\to\infty$.

For $q\in L^{\infty}(\Omega )$, we denote by $\sigma_1 (q,\Omega )$ 
the first eigenvalue of $-\Delta +q$ over $\Omega$ 
with the homogeneous Neumann boundary condition. 
As is well known, the following properties hold:
\begin{itemize}
\item[(i)]
The mapping $q\mapsto \sigma_1 (q,\Omega ):L^{\infty}(\Omega )\to \mathbb{R}$ is continuous.
\item[(ii)]
$\sigma_1 (0,\Omega )=0$.
\item[(iii)]
If $q_1 \ge q_2 \mbox{ and }q_1 \not\equiv q_2$, then 
$\sigma_1 (q_1 ,\Omega )>\sigma_1 (q_2 ,\Omega )$.
\end{itemize}
In addition, we denote by $\sigma_1^D (\Omega_0 )$ 
the first eigenvalue of $-\Delta$ over $\Omega_0$ 
with the homogeneous Dirichlet boundary condition. 

The curve $\lambda=\sigma_1(b\mu\1_{\Omega_1},\Omega)$
for $\mu\ge 0$ represents the set of parameter values 
$(\lambda,\mu)$ where positive solutions bifurcate from
the semitrivial solution $(0,\mu)$; this fact will be established in Theorem \ref{exthm1} below.
The following lemma, derived from \cite[Theorem~2.1]{DS} (see also \cite[Lemma~2.1]{Oeda2011}), shows that the curve $\lambda=\sigma_1(b\mu\1_{\Omega_1},\Omega)$ 
is monotone increasing and uniformly bounded for $\mu\ge 0$, see Figure 2.
\begin{lem}\label{du-shi}
The function $\mu\mapsto\sigma_1(b\mu\1_{\Omega_1},\Omega)$ is continuous and strictly
increasing for $\mu\ge 0$, and satisfies
\[
\sigma_1(0,\Omega)=0,\quad 
\sigma_1(b\mu\1_{\Omega_1},\Omega)<b\mu\ \text{for all }\mu>0,\quad
\lim_{\mu\to\infty}\sigma_1(b\mu\1_{\Omega_1},\Omega)=\sigma_1^{D}(\Omega_0).
\]
\end{lem}
It should be remarked that the boundedness of 
$\lim_{\mu\to\infty}\sigma_1(b\mu\1_{\Omega_1},\Omega)$
is a consequence of the presence of the protection zone $\Omega_{0}$, 
since $\sigma_{1}^{D}(\Omega_{0})\to\infty$ as $|\Omega_{0}|\to 0$. 
We also note that $\sigma_1(b\mu\1_{\Omega_1},\Omega)$ is independent of the nonlinear 
diffusion parameter~$\alpha$.
In what follows, we denote $\sigma_{1}(b\mu\1_{\Omega_{1}},\Omega)$ 
simply by $\sigma_{1}(b\mu\1_{\Omega_{1}})$.
Furthermore, the associated positive eigenfunction 
with $L^2$ normalization will be denoted by $\phi^*$, that is,
\begin{equation}\label{phistar1}
\begin{cases}
-\Delta \phi^* +b\mu\1_{\Omega_1} \phi^* =\sigma_1 (b\mu\1_{\Omega_1} )\phi^*,\quad \phi^{*}>0 \ \ \mbox{in}\ \Omega ,\\
\partial_n \phi^* =0\ \ \mbox{on}\ \partial \Omega ,\quad \|\phi^{*}\|_{L^2(\Omega)}=1.
\end{cases}
\end{equation}
We next turn to the stationary problem and describe the structure of 
positive solutions.  
We regard $\lambda$ as the bifurcation parameter and introduce two 
branches of semitrivial solutions,
\begin{equation}\label{gammauv}
\varGamma_u=\{(\lambda,u,v)=(\lambda,\lambda,0):\lambda>0\},
\qquad
\varGamma_v=\{(\lambda,u,v)=(\lambda,0,\mu):\mu>0\}.
\end{equation}
We are now in position to state the existence theorem for positive 
solutions of \eqref{sp2}.
\begin{thm}\label{exthm1}
Suppose that $\mu \ge 0$ is fixed. Then \eqref{sp2} has at least one positive solution 
for any $\lambda >\sigma_1 (b\mu\1_{\Omega_1} )$. 
More specifically, for any fixed $\mu >0$, positive solutions of \eqref{sp2} bifurcate from $\varGamma_v$ 
if and only if $\lambda =\sigma_1 (b\mu\1_{\Omega_1} )$
in the sense that
all positive solutions of \eqref{sp2} 
near $(\sigma_1 (b\mu\1_{\Omega_1} ),0,\mu )\in \mathbb{R}\times X$ 
can be expressed as
$$
\varGamma_{1, \delta} =\left\{ (\lambda ,u,v)=
\left( \lambda (s),s(\phi^* +\widetilde{u}(s) ),\mu +s(\psi^* +\widetilde{v}(s))\right) :s\in (0,\delta )\right\}
$$
for some $\delta >0$. Here $\psi^*=\psi^{*}(\alpha )$ is given by
\begin{equation}\label{psistar1}
\psi^* :=(-\Delta +\mu )^{-1}_{\Omega_1}
\left( 
  \mu \left[ 
    \alpha \left\{ \sigma_1(b\mu\1_{\Omega_1}) - b\mu \right\} + c
  \right] \phi^*
\right),
\end{equation}
where $(-\Delta +\mu )^{-1}_{\Omega_1}$ is the inverse operator of $-\Delta +\mu$ 
over $\Omega_1$ subject to the homogeneous Neumann boundary condition.
The functions
$(\lambda (s),\widetilde{u}(s),
\widetilde{v}(s))$ depend smoothly on $s$ 
and satisfy $(\lambda (0),\widetilde{u}(0),\widetilde{v}(0))=
(\sigma_1 (b\mu\1_{\Omega_1} ),0,0)$, $\int_{\Omega}\widetilde{u}(s)\phi^* dx=0$ for any $s\in (0,\delta)$.
Moreover, 
$
\varGamma_{1, \delta}$
can be extended as an unbounded connected set 
$\varGamma_1 (\subset \mathbb{R}\times X)$
of positive solutions of
\eqref{sp2} in direction to $\lambda\to\infty$,
i.e.
$$
\{ \lambda:(\lambda ,u,v)\in \varGamma_1 \}\supset 
\left( \sigma_1 (b\mu\1_{\Omega_1} ),\infty \right) .
$$
Furthermore, there exists a positive constant 
$\alpha^{*}=\alpha^{*}(\mu ,b\1_{\Omega_1}, c,\Omega ,\Omega_0 )$ such that 
the bifurcation from $\varGamma_v$ at 
$\lambda =\sigma_1 (b\mu\1_{\Omega_1} )$ is supercritical
($\lambda'(0)>0$) if $\alpha <\alpha^{*}$, 
and subcritical
($\lambda'(0)<0)$ if $\alpha >\alpha^{*}$.
\end{thm}
\begin{figure}
\centering
\includegraphics*[scale=.9]{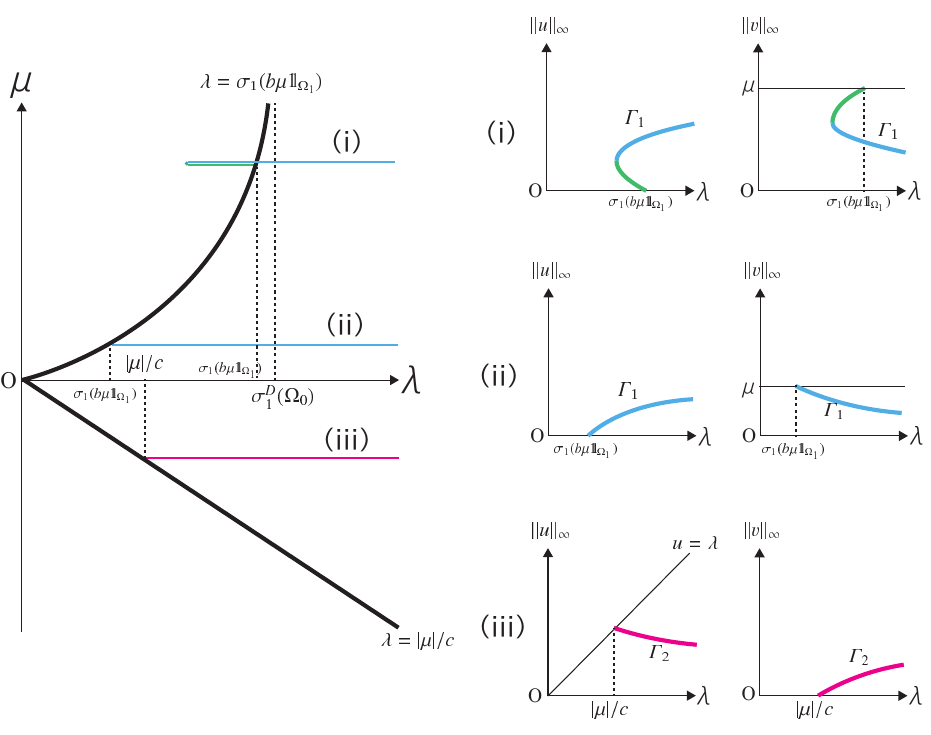}
\caption{
(Left) Graphs of $\lambda=\sigma_{1}(b\mu\1_{\Omega_{1}})$ and $\lambda=|\mu|/c$ in the $\lambda\mu$ plane. (Right) Typical bifurcation diagrams: (i) the case $\alpha>\alpha^{*}$ in Theorem \ref{exthm1}; (ii) the case $\alpha<\alpha^{*}$ in Theorem \ref{exthm1}; (iii) the case in Theorem \ref{exthm2}.
}
\label{fig2}
\end{figure}

By Theorem \ref{exthm1}, when $\mu>0$ and $\alpha>\alpha^{*}$, the curve 
$\varGamma_{1,\delta}$, which initially bifurcates in the direction of decreasing 
$\lambda$ from the bifurcation point $\lambda=\sigma_{1}(b\mu\1_{\Omega_{1}})$, eventually 
turns around and extends toward $\lambda\to\infty$.  
This indicates that the global continuum $\varGamma_{1}$ has a saddle-node-type bifurcation (turning point), 
as illustrated by the typical branch shown in Figure 1(i).
Consequently, for values of $\lambda$ slightly below the bifurcation point 
$\sigma_{1}(b\mu\1_{\Omega_{1}})$, \eqref{sp2} admits at least two 
positive solutions.
From a biological perspective, this multiplicity of coexistence steady states of \eqref{P}
occurs only through the combined effects of the protection zone $\Omega_{0}$ and the 
directed population flux $\alpha$.  
Neither mechanism alone is sufficient to generate this type of bistability; it is 
precisely their interaction that enables the emergence of multiple stable coexistence configurations.

On the other hand, in the case $\mu<0$, the line $\lambda=|\mu|/c$
represents the set of parameter values $(\lambda,\mu)$ at which positive 
solutions bifurcate from the other semitrivial solution $(\lambda,0)$,
see Figure 2(iii).
\begin{thm}\label{exthm2}
Suppose that $\mu <0$ is fixed. Then \eqref{sp2} has at least one positive solution 
for any $\lambda >|\mu |/c$. 
To be more specific, 
positive solutions of \eqref{sp2} bifurcate from $\varGamma_u$ 
if and only if $\lambda =|\mu |/c$. 
In addition, all positive solutions of \eqref{sp2} 
near $(|\mu |/c,\lambda ,0)\in \mathbb{R}\times X$ 
can be expressed as
$$
\varGamma_{2, \delta} =\left\{ (\lambda ,u,v)=
\left( \lambda (s),\lambda +s(\phi_* +\widetilde{u}(s)),s(1+\widetilde{v}(s))\right) :s\in (0,\delta )\right\}
$$
for some $\delta >0$. Here 
$\phi_*$ is given by
\begin{equation}\label{phistar2}
\phi_* =(-\Delta +\lambda )^{-1}_{\Omega}[-b\1_{\Omega_1}\lambda ],
\end{equation}
where $(-\Delta +\lambda )^{-1}_{\Omega}$ is the inverse operator of $-\Delta +\lambda$ 
over $\Omega$ subject to the homogeneous Neumann boundary condition.
The functions
$(\lambda (s),\widetilde{u}(s),\widetilde{v}(s))$ depend smoothly on
$s$ 
and satisfies $(\lambda (0),\widetilde{u}(0),\widetilde{v}(0))=(|\mu |/c,0,0)$,
$\int_{\Omega_1}\widetilde{v}(s)dx=0$ for any $s\in (0, \delta )$.
Moreover,
$\varGamma_{2,\delta}$
can be extended as a global connected set
$\varGamma_2 (\subset \mathbb{R}\times X)$ of positive solutions of \eqref{sp2} satisfying
$$
\{ \lambda:(\lambda ,u,v)\in \varGamma_2 \}\supset (|\mu |/c,\infty ).
$$
Furthermore, the bifurcation from $\varGamma_u$ at $\lambda =|\mu |/c$ 
is supercritical 
($\lambda'(0)>0$)
for any $\alpha$.
\end{thm}
Next, we present results on the asymptotic behavior of positive solutions of
\eqref{sp2} as the directed population flux parameter $\alpha$ tends to infinity.
\begin{thm}\label{abthm}
Let $(u_{\alpha},v_{\alpha})$ be any positive solution of \eqref{sp2} 
for each $\alpha \ge 0$.
\begin{enumerate}[{\rm (i)}]
\item
Suppose that $(\lambda ,\mu )$ satisfies
$$
{\rm (A)}\ \ \lambda >\sigma_1 (b\mu\1_{\Omega_1}),\ \mu \ge 0
\quad \mbox{or}\quad {\rm (B)}\ \ \mu <0<\lambda .
$$
Then
$$
\lim_{\alpha \to \infty}(u_{\alpha},v_{\alpha})=(\lambda ,0)\ 
\mbox{in}\ C^1 (\overline{\Omega})\times C^1 (\overline{\Omega}_1 ).
$$
\item
Suppose that $(\lambda ,\mu )$ satisfies 
$0<\lambda \le \sigma_1 (b\mu\1_{\Omega_1})$ and $\mu >0$. 
Then there exists a sequence $\{ (u_{\alpha_k},v_{\alpha_k})\}_{k=1}^{\infty}$ 
with $\lim_{k\to \infty}\alpha_k =\infty$ such that either (a) or (b) occurs:
\begin{enumerate}[{\rm (a)}]
\item
$\displaystyle\lim_{k \to \infty}(u_{\alpha_k},v_{\alpha_k})=(\lambda ,0)\ 
\mbox{in}\ C^1 (\overline{\Omega})\times C^1 (\overline{\Omega}_1 )$.
\item
$\displaystyle\lim_{k \to \infty}(\alpha_k u_{\alpha_k},v_{\alpha_k})=(w_{\infty},v_{\infty})\ 
\mbox{in}\ C^1 (\overline{\Omega})\times C^1 (\overline{\Omega}_1 )$. \\
Here $(w_{\infty},v_{\infty})$ is a positive solution of
\begin{equation}\label{LP2}
\begin{cases}
\Delta w+w(\lambda -b\1_{\Omega_1}v)=0\ \ &\mbox{in}\ \Omega,\\
\Delta v+\dfrac{v}{1+w}\left\{ w(\lambda -bv)+\mu -v\right\} =0
\ \ &\mbox{in}\ \Omega_1 ,\\
\partial_n w=0\ \
\ \ &\mbox{on}\ \partial \Omega,\\
\partial_n v=0\ \
\ \ &\mbox{on}\ \partial \Omega_1
\end{cases}
\end{equation}
if $\lambda <\sigma_1 (b\mu\1_{\Omega_1})$, and 
$(w_{\infty},v_{\infty})=(0,\mu )$ if $\lambda =\sigma_1 (b\mu\1_{\Omega_1})$.
\end{enumerate}
\end{enumerate}
\end{thm}
According to Theorem \ref{abthm}, when the directed population flux 
$\alpha$ representing the tendency of predator to move toward regions with 
abundant prey is taken to be very large, every coexistence steady state of \eqref{P} asymptotically collapses to the semitrivial state $(\lambda,0)$. 
In biological terms, this means that, although a strong directed movement toward prey might appear advantageous for the predator, an excessively large flux ultimately drives the predator to extinction while the prey survives alone.
It should also be noted that, in parameter regimes $(\lambda,\mu)$ that lie to the left of the curve $\lambda=\sigma_{1}(b\mu\,\1_{\Omega_{1}})$ 
and, at the same time, admit multiple positive steady states, 
an additional phenomenon may occur. Besides convergence to the semitrivial state $(\lambda,0)$, the system may alternatively approach a small positive prey-only equilibrium of order $1/\alpha$, which is characterized by the 
limiting problem \eqref{LP2}. This highlights a subtle interplay between directed 
movement and the refuge structure, producing biologically intriguing 
alternative asymptotic states.

The following theorem shows the global bifurcation structure of positive solutions of
the limiting system \eqref{LP2}, see Figure 3.
\begin{figure}
  \centering
  \includegraphics[width=\textwidth]{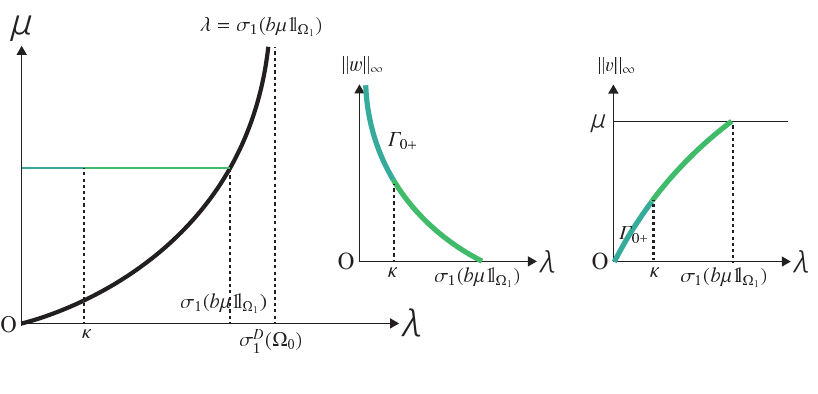}
  \caption{(Left) Graph of $\lambda = \sigma_{1}(b\mu\1_{\Omega_{1}})$ in the $\lambda\mu$ 
  plane. (Right) Typical bifurcation diagram expected in Theorem 3.5.}
  \label{fig3}
\end{figure}

\begin{thm}\label{LP2thm}
Suppose that $\mu >0$. 
Positive solutions of \eqref{LP2} with parameter $\lambda$ bifurcate from 
$(w,v)=(0,\mu )$ if and only if $\lambda =\sigma_1 (b\mu\1_{\Omega_1} )$. 
Moreover, in $\mathbb{R}\times X$, the connected set $\varGamma$ of positive solutions of \eqref{LP2} 
which bifurcates from $(\lambda ,w,v)=(\sigma_1 (b\mu\1_{\Omega_1} ),0,\mu )$ is 
uniformly bounded for $(\lambda , v)$ and unbounded with respect to $w$. 
Furthermore, 
$\varGamma$ contains an unbounded curve $\varGamma_{0+}$ 
parameterized by $\lambda\in (0,\kappa )$ with some small 
$\kappa\in (0, \sigma_1 (b\mu\1_{\Omega_1} ))$
as follows:
\begin{align*}
&\varGamma_{0+}=\\
&\biggl\{(\lambda, w(\lambda ), v(\lambda))\in (0,\kappa)\times X\,:\, 
w(\lambda )=\dfrac{s(\lambda )}{\lambda}+\phi (\lambda ),\ 
v(\lambda )=\lambda (\,t(\lambda )+\lambda \psi (\lambda )\,)\biggr\},
\end{align*}
where
$(s(\lambda ), t(\lambda ), \phi (\lambda ), \psi (\lambda ))
\in \mathbb{R}^{2}\times X$
are continuously differentiable functions for
$\lambda\in [0,\kappa)$
satisfying
$\int_{\Omega}\phi (\lambda )\,dx=
\int_{\Omega_1}\psi (\lambda )\,dx=0$
for all $\lambda\in [0,\kappa)$,
and
\[
(s(0), t(0))=\biggl(
\mu\dfrac{|\Omega_1|}{|\Omega_0|},
\dfrac{|\Omega |}{b|\Omega_1|}\biggr),\ \ 
\phi(0)=\mu (-\Delta)^{-1}_{\Omega,\#}
\biggl(\dfrac{|\Omega_1|}{|\Omega_0|}\,\1_{\Omega_0}-
\1_{\Omega_1}\biggr),\ \ 
\psi (0)=0,\]
where $(-\Delta )^{-1}_{\Omega,\sharp}$ denotes the inverse of 
$-\Delta$ on $\Omega$ under the homogeneous Neumann boundary condition on 
$\partial\Omega$, restricted to the subspace of functions with zero mean value.
\end{thm}

\begin{figure}
  \centering
  \includegraphics[width=\textwidth]{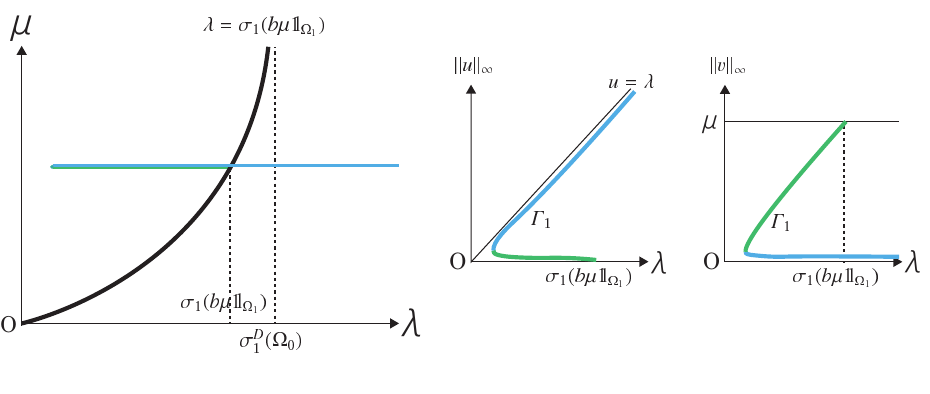}
  \caption{A conjectured bifurcation diagram of positive solutions of \eqref{sp} in the regime of large directional movement strength.}
  \label{fig4}
\end{figure}
Theorem \ref{LP2thm} shows that positive solutions of \eqref{LP2} bifurcate from the
same point $\lambda=\sigma_{1}(b\mu\,\1_{\Omega_1})$ as in the original
system \eqref{sp2}, but for the limiting problem the bifurcating branch extends
in the direction $\lambda\to 0^{+}$.  
When $\lambda>0$ is sufficiently small, the branch parametrized by $\lambda$
exhibits the following asymptotic behavior:  
the $v$-component converges to $0$, whereas the $w$-component grows at
order $1/\lambda$.  
In this sense, $\lambda=0$ plays the role of a bifurcation from
infinity in the limiting problem.
With the scaling $\alpha u\to w$ in mind, this suggests the following
interpretation for the original system \eqref{sp2}.  
When the directed population flux $\alpha$ is large, the saddle-node
bifurcation appearing in \eqref{sp2} can be viewed as a perturbation of the
bifurcation-from-infinity structure observed in the limit problem.  
From a biological perspective, as $\lambda\to 0$, the predator density
$v$ tends toward extinction, while the rescaled prey density $w$ becomes
large, see Figure 4.
Thus, the limit problem offers a useful viewpoint for understanding how
strong directed movement modifies the structure of coexistence steady states in
the original model \eqref{P}.

In our previous papers \cite{OK2018, KO2021}, we investigated the stationary problem of \eqref{P} in the absence of a protection zone (i.e., $\Omega_{0}=\emptyset$) under Dirichlet boundary conditions.
It was shown that, for large values of $\alpha$, the coexistence region in the $\lambda\mu$ plane expands and that the global bifurcation branch of steady states exhibits a turning point.
However, in such a Dirichlet problem, the minimum growth rate required for prey survival (corresponding to the turning point) decreases only marginally, even in the limit $\alpha \to \infty$.
More recently, for \eqref{P} without a protection zone, Wang \cite{Wang2026}
analyzed the stationary problem under Neumann boundary conditions with spatially heterogeneous coefficients $b$ and $c$.
Her results indicate that the effect of $\alpha>0$ does not lead to a substantial expansion of the coexistence region.
In contrast, our present results suggest that the pronounced expansion of the coexistence region in the $\lambda\mu$ plane arises from a synergistic effect between the directed population flux parameter 
$\alpha$ and the presence of the protection zone $\Omega_{0}$.

In \cite{LGMH20,LGMH24}, the effects of spatial heterogeneities on the bifurcation branches
of coexistence steady states in saturation predator--prey models were investigated,
and bifurcation branches with a similar qualitative shape were also observed.
However, in that setting, additional bifurcation points appear along those branches,
depending on the spatial modes, as shown in \cite{KLMH26}.

\section{Proofs of Theorems \ref{exthm1} and \ref{exthm2}}
\subsection{A priori estimates and the nonexistence region for
positive stationary solutions}
As a beginning of this section, we note fundamental but useful inequalities for deriving
a priori estimates of positive solutions of \eqref{sp}:
If $q>0$ and $r,\,s\in\mathbb{R}$, then
\begin{equation}\label{qrs}
\min \left\{ \frac{r}{q},\, s\right\} \le\frac{r\xi +s}{q\xi +1}\le\max \left\{ \frac{r}{q},\, s\right\} 
\end{equation}
for any $\xi \ge 0$.
The following lemma gives a uniform $W^{2,p}$ bound of any positive solution of \eqref{sp}
independent of $\alpha$.
\begin{lem}\label{ape}
Let $(u,v)$ be any positive solution of \eqref{sp}. Then, for any $p>N$ and $\theta \in (0,1)$, 
there exist positive constants $C_i \, (i=1,2,3,4)$ independent of $\alpha$ such that
$$
\| u \|_{W^{2,p}(\Omega )}\le C_1 ,\quad \| v \|_{W^{2,p}(\Omega_1 )}\le C_2
$$
and
$$
\| u\|_{C^{1,\theta}(\overline{\Omega})}\le C_3 ,\quad 
\| v\|_{C^{1,\theta}(\overline{\Omega}_1 )}\le C_4 .
$$
\end{lem}
\begin{proof}
Since $u$ satisfies $\Delta u+u(\lambda -u)\ge 0$ in $\Omega$, 
a standard comparison argument yields
\begin{equation}\label{pfape1}
\max_{\overline{\Omega}}u\le \lambda .
\end{equation}
Let $x_{0}\in\overline{\Omega}_{1}$ be a maximum point of $v$, so that
$v(x_{0})=\max_{\overline{\Omega}_{1}}v$.
Then
applying the maximum principle (e.g., \cite[Proposition 2.1]{LN1}) 
to the second equation of \eqref{sp2}, we have
\begin{equation}\label{pfape2}
\alpha u(x_0 )(\lambda -u(x_0)-bv(x_0 ))+\mu +cu(x_0 )-v(x_0 )\ge 0,
\end{equation}
where $v(x_0)=\max_{\overline{\Omega}_1}v$ with $x_0 \in \overline{\Omega}_1$. 
By \eqref{qrs}-\eqref{pfape2}, we obtain
\begin{equation}\label{pfape3}
\begin{split}
v(x_0 )&\le \frac{\alpha u(x_0 )(\lambda -u(x_0))+\mu +cu(x_0 )}{b\alpha u(x_0 )+1}\\
&<\frac{\lambda \alpha u(x_0 )+\mu +c\lambda}{b\alpha u(x_0 )+1}
\le \max \left\{ \frac{\lambda}{b},\, \mu +c\lambda \right\} .
\end{split}
\end{equation}
By \eqref{qrs}, we also have
\begin{equation}\label{pfape4}
\begin{split}
\min \{ \lambda -u-bv,\, \mu +cu-v\}&\le 
\dfrac{(\lambda -u-bv)\alpha u+\mu +cu-v}{\alpha u+1}\\
&\le \max \{ \lambda -u-bv,\, \mu +cu-v\} .
\end{split}
\end{equation}
Thus we see from \eqref{pfape1}, \eqref{pfape3} and \eqref{pfape4} 
that for any $p>N$, there exist two positive constants 
$\tilde{C}_1$ and $\tilde{C}_2$ independent of $\alpha$ such that
$$
\| u(\lambda -u-b\1_{\Omega_1}v)\|_{L^p (\Omega )}+\| u\|_{L^p (\Omega )}\le \tilde{C}_1
$$
and
$$
\left\| \frac{v}{1+\alpha u}\left\{ \alpha u(\lambda -u-bv)+\mu +cu-v\right\} \right\|
_{L^p (\Omega_1 )}+\| v\|_{L^p (\Omega_1 )}\le \tilde{C}_2 .
$$
Therefore, the conclusion of Lemma \ref{ape} follows 
from elliptic regularity theory and the Sobolev embedding theorem.
\end{proof}
Building upon the a priori estimate for positive solutions 
obtained in Lemma \ref{ape},
we next employ this bound as a basis 
for deriving sufficient conditions ensuring
the nonexistence of positive solutions. 
In particular, our aim is to characterize
a sufficient nonexistence region in the $(\lambda,\mu)$-plane 
and to investigate
its dependence on $\alpha$. Identifying such a nonexistence
(sufficient) region provides important guidance for the subsequent analysis 
of
the existence of positive solutions and 
of the global bifurcation structure.

To derive the boundary curve that determines 
the sufficient nonexistence region, we first establish the following lemma,
see also Figure 5.
\begin{lem}\label{implicitlem}
Define the function $K(\lambda, \mu, \alpha)$ by
\[
K(\lambda, \mu, \alpha ):=
\lambda-\sigma_{1}\left(
\dfrac{b(c\lambda +\mu )}{\alpha b\lambda +1}\1_{\Omega_1}\right).
\]
Then for any $\alpha >0$ and $\mu>c/(\alpha b)$,
there exists a unique $\ell (\mu, \alpha )>0$ such that
\[
K(\ell (\mu, \alpha ), \mu, \alpha )=0.\]
Furthermore, $\ell (\mu, \alpha )$ is continuously differentiable
for 
$(\mu, \alpha )\in (c/(\alpha b), \infty)\times (0, \infty )$
and satisfies the following properties:
\begin{align*}
&\dfrac{\partial\ell}{\partial\mu}(\mu, \alpha )>0
\quad\mbox{for all}\ 
(\mu, \alpha )\in \left(\dfrac{c}{\alpha b},\infty\right)\times
(0,\infty),\\
&\ell(\mu, \alpha)\searrow\sigma_{1}\left(\dfrac{c}{\alpha}\1_{\Omega_1}\right)\quad\mbox{as}\ \mu\searrow\dfrac{c}{\alpha b},\\
&\ell(\mu, \alpha )\nearrow\sigma_{1}^{D}(\Omega_0)\quad\mbox{as}\ 
\mu\nearrow\infty\ \mbox{for any}\ \alpha>0,\\
&\ell(\,\cdot\,,\alpha)\searrow 0\quad
\mbox{uniformly in any compact subset of $(0,\infty)$ as 
$\alpha\nearrow\infty$.}
\end{align*}
\end{lem}

\begin{proof}
By virtue of Lemma \ref{du-shi},
the function $\sigma_{1}(\xi \1_{\Omega_1})$ is monotone increasing
for $\xi\in (0,\infty )$ and satisfies
\begin{equation}\label{sigmalim}
\lim_{\xi\to 0^+}\sigma_1(\xi\1_{\Omega_1})=0
\quad\mbox{and}\quad
\lim_{\xi\to\infty}\sigma_1(\xi \1_{\Omega_1})=
\sigma^{D}_{1}(\Omega_0).
\end{equation}
In addition, we note that
$\sigma_{1}(\xi\1_{\Omega_1})$ is continuously differentiable
for $\xi>0$.
The proof of the $C^{1}$-regularity can be found, for instance, in 
Yamada~\cite[Proposition 1.1\,(iii)]{Yam}. 
In his work, a continuous function, depending continuously 
on the parameter, 
is introduced as a potential function. Nevertheless,
the proof remains valid in the above setting,
where the potential is given 
by the step function $\xi\1_{\Omega_1}$.
It follows from the chain rule that
$K(\lambda, \mu, \alpha )$ is continuously differentiable.
Since
\[
\dfrac{\partial}{\partial\lambda}\left(
\dfrac{b(c\lambda +\mu)}{\alpha b\lambda +1}\right)
=\dfrac{b(c-\mu\alpha b)}{(\alpha b\lambda +1)^{2}}
\begin{cases}
\ge 0\quad&\mbox{if}\ \mu\le c/(\alpha b),\\
<0\quad&\mbox{if}\ \mu>c/(\alpha b).
\end{cases}
\]
Then for any fixed $\alpha >0$ and $\mu>c/(\alpha b)$,
\[
[0,\infty)\ni\lambda\mapsto K(\lambda, \mu, \alpha )
=\lambda-\sigma_{1}\left(\dfrac{b(c\lambda+\mu )}{\alpha b\lambda+1}
\1_{\Omega_1}\right)
\]
is monotone increasing and satisfies
\[
K(0, \mu, \alpha )=-\sigma_{1}(b\mu\1_{\Omega_1})<0
\quad
\mbox{and}\quad
\lim_{\lambda\to\infty }K(\lambda, \mu, \alpha )=\infty.
\]
Then, the intermediate value theorem gives $\ell (\mu, \alpha)>0$ such that
$K(\lambda,\mu,\alpha)<0$
for $0<\lambda<\ell (\mu, \alpha )$;
\[
K(\ell(\mu, \alpha ), \mu, \alpha )=0;
\]
and
$K(\lambda,\mu,\alpha)>0$
for $\lambda>\ell (\mu, \alpha )$.
By virtue of the implicit function theorem,
one can verify that
$\ell (\mu, \alpha )$ is continuously differentiable
with respect to $(\mu, \alpha)\in (c/(\alpha b), \infty )\times (0,\infty)$.
It is noted that $K(\lambda, \mu, \alpha )$ is monotone decreasing
with respect to $\mu>c/(\alpha b)$.
By the implicit differentiation formula, we see 
\begin{equation}\label{muinc}
\frac{\partial \ell}{\partial \mu}(\mu,\alpha)
= -\,\frac{K_{\mu}\bigl(\ell(\mu,\alpha),\mu,\alpha\bigr)}
         {K_{\lambda}\bigl(\ell(\mu,\alpha),\mu,\alpha\bigr)}
>0
\quad\mbox{for any}\ (\mu,\alpha)\in
\left(\dfrac{c}{\alpha b},\infty\right)\times(0,\infty).
\end{equation}
By \eqref{sigmalim}, we note that
$
K(\lambda, \mu, \alpha )\searrow
\lambda-\sigma_{1}^{D}(\Omega_0)$
as
$\mu\nearrow\infty$.
From this fact, it is possible to verify that
$
\ell (\mu, \alpha)\nearrow \sigma_{1}^{D}(\Omega_{0})$
as
$\mu\nearrow\infty$.
Since
\[
K\left(\lambda, \dfrac{c}{\alpha b}, \alpha\right)
=
\lambda-\sigma_{1}\left(\dfrac{c}{\alpha}\1_{\Omega_1}\right),
\]
then we know that
$\ell (c/(\alpha b), \alpha )=\sigma_{1}(c\1_{\Omega_1}/\alpha )$.
Recalling that $\sigma_{1}(\xi\1_{\Omega_{1}})$ is monotone increasing 
with respect to $\xi$, it follows from the chain rule that 
$K(\lambda,\mu,\alpha)$ is monotone increasing for $\alpha >0$. Hence, by implicit 
differentiation one obtains
\[
\frac{\partial \ell}{\partial \alpha}(\mu,\alpha)
= -\,\frac{K_{\alpha}\bigl(\ell(\mu,\alpha),\mu,\alpha\bigr)}
         {K_{\lambda}\bigl(\ell(\mu,\alpha),\mu,\alpha\bigr)}<0 
\quad
\mbox{for any}\ (\mu, \alpha)\in
\left(\dfrac{c}{\alpha b},\infty\right)\times(0,\infty).
\]
Finally we show the uniform convergence of $\ell(\,\cdot\,,\alpha )$
to zero
on any compact set of $(0,\infty)$ as $\alpha\to\infty$.
From \eqref{muinc} it follows that $\ell(\mu,\alpha)$ is monotone increasing 
with respect to $\mu$. Hence, for any fixed $M>0$, it suffices to show that 
$\ell(M,\alpha)\to 0$ as $\alpha\to\infty$.
To prove by contradiction, assume that 
\[
\ell_{\infty}:=\limsup_{\alpha\to\infty}\ell(M,\alpha)>0.
\]
Then there exists a sequence $\{\alpha_{j}\}$ with $\alpha_{j}\to\infty$ 
as $j\to\infty$ such that $\ell(M,\alpha_{j})\to\ell_{\infty}$.
From
\[
\ell(M,\alpha_{j})
=\sigma_{1}\!\left(
\frac{b\,(c\,\ell(M,\alpha_{j})+M)}{\alpha_{j}b\,\ell(M,\alpha_{j})+1}\1_{\Omega_1}
\right),
\]
letting $j\to\infty$ on both sides yields
\[
\lim_{j\to\infty}\ell(M,\alpha_{j})=\sigma_{1}(0)=0,
\]
which contradicts the assumption.
The proof of Lemma \ref{implicitlem} is complete.
\end{proof}

We define a function $\widetilde{\ell}(\mu,\alpha)$ by connecting the function 
$\ell(\mu,\alpha)$ introduced in Lemma~\ref{implicitlem} with 
$\sigma_{1}(b\mu\1_{\Omega_1})$ as follows:
\[
\widetilde{\ell}(\mu,\alpha)=
\begin{cases}
\sigma_{1}(b\mu\1_{\Omega_{1}}) & \mbox{for}\ 0<\mu\le c/(\alpha b),\\[6pt]
\ell(\mu,\alpha) & \mbox{for}\ \mu>c/(\alpha b).
\end{cases}
\]
\begin{figure}
  \centering
  \includegraphics[width=0.5\textwidth]{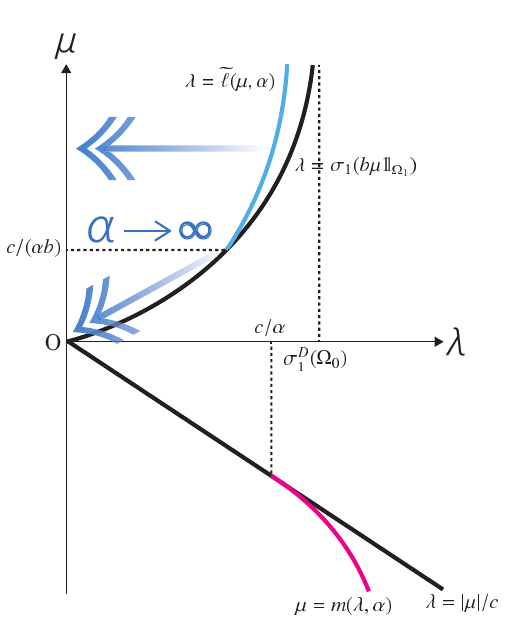}
  \caption{
  Graphs of $\lambda=\widetilde{\ell} (\mu, \alpha)$ and $\mu=m(\lambda, \alpha)$ in the $\lambda\mu$ plane in 
  Propositions \ref{nonexprop1} and \ref{nonexprop2}.}
  \label{fig5}
\end{figure}
In this case, by Lemmas~\ref{du-shi} and \ref{implicitlem}, it follows 
that for any fixed $\alpha>0$, the function $\widetilde{\ell}(\mu,\alpha)$ 
is monotone increasing and continuous with respect to $\mu>0$ 
(in particular, it is continuous also at $\mu=c/(\alpha b)$), and satisfies 
the following properties:
\begin{itemize}
  \item $\widetilde{\ell}(0,\alpha)=0$, and 
  $\widetilde{\ell}(\mu,\alpha)\nearrow\sigma_{1}^{D}(\Omega_{0})$ 
  as $\mu\nearrow\infty$;
  \item $\widetilde{\ell}(\,\cdot\,,\alpha)\searrow 0$ uniformly on any 
compact subset of $(0,\infty)$ as $\alpha\nearrow\infty$.
\end{itemize}
The following proposition provides a necessary condition for the existence 
of positive solutions of \eqref{sp} in the range $\mu>0$, that is, 
a sufficient condition for the nonexistence of positive solutions.

\begin{prop}\label{nonexprop1}
If $\alpha >0$ and $0<\lambda\le\widetilde{\ell}(\mu, \alpha)$, 
then \eqref{sp2}, equivalently \eqref{sp}, admits no positive solution.
\end{prop}

\begin{proof}
For any positive solution $(u,v)$ to \eqref{sp2},
let $x_{1}\in\overline{\Omega}_{1}$ be the minimum point of
$v$, that is,
$v(x_{1})=\min_{\overline{\Omega}_1}v$.
Then the usual application of the maximum principle
as in the proof of Lemma \ref{ape}, we obtain
\begin{equation}\label{minv}
\min_{\overline{\Omega}_{1}}v=
v(x_1)\ge
\dfrac{\alpha u(x_1)(\lambda -u(x_1))+\mu+cu(x_1)}
{\alpha b u(x_1)+1}=:r(u(x_1)),
\end{equation}
where 
\[
r(s):=\dfrac{\alpha s(\lambda -s)+\mu+cs}{\alpha bs +1}.\]
We recall $0<u(x_1)\le\lambda $ by \eqref{pfape1}.
It is easily checked that
\[
\min_{0\le s\le\lambda}r(s)=\min\{
r(0), r(\lambda )\}=
\min\left\{
\mu, \dfrac{\mu+c\lambda}{\alpha b\lambda +1}\right\}.\]

First, consider the case $0\le\mu\le c/(\alpha b)$. 
Hence in this case,
\[
\min_{0\le s\le\lambda} r(s)=r(0)=\mu.
\]
Then \eqref{minv} implies that 
\[
\min_{\overline{\Omega}_{1}} v=v(x_1) \ge \mu.
\]
Next, consider the first equation of \eqref{sp2}:
\begin{equation}\label{firstpote}
\begin{cases}
-\Delta u+(u+b\1_{\Omega_1}v)u=\lambda u,\quad u>0
& \mbox{in } \Omega,\\
\partial_{n}u=0
& \mbox{on } \partial\Omega.
\end{cases}
\end{equation}
Then it follows that $\lambda=\sigma_{1}(u+b\1_{\Omega_1}v)$.
By the monotone increasing property of $\sigma_{1}(\,\cdot\,)$,
we see that
\[
\lambda=\sigma_{1}(u+b\1_{\Omega_1}v)>
\sigma_{1}\bigl(bv(x_1)\1_{\Omega_1}\bigr)\ge
\sigma_{1}(b\mu\1_{\Omega_1}).
\]
Therefore, if $0<\mu\le c/(\alpha b)$ and 
$\lambda\le\sigma_{1}(b\mu\1_{\Omega_1})$, 
then \eqref{sp2}, and hence also \eqref{sp}, 
admits no positive solution.

Next, consider the case $\mu>c/(\alpha b)$. 
In this case, 
\[
\min_{0\le s\le\lambda} r(s)
=r(\lambda)
=\dfrac{\mu+c\lambda}{\alpha b\lambda+1}.
\]
It follows from \eqref{minv} that
\[
\min_{\overline{\Omega}_{1}} v =v(x_1)\;\ge\; 
\dfrac{\mu+c\lambda}{\alpha b\lambda+1}.
\]
Therefore, by \eqref{firstpote} we obtain
\[
\lambda=\sigma_{1}(u+b\1_{\Omega_1}v)>
\sigma_{1}\!\left(
\dfrac{b(\mu+c\lambda )}{\alpha b\lambda+1}\1_{\Omega_1}
\right).
\]
Consequently, if $\mu>c/(\alpha b)$ and 
$\lambda\le \ell(\mu,\alpha)$, then \eqref{sp2}, and hence also 
\eqref{sp}, admits no positive solution.
We complete the proof of Proposition \ref{nonexprop1}.
\end{proof}
We remark that in the case of $\alpha =0$, \eqref{sp2} has no positive solution if $\lambda \le \sigma_{1}(b\mu\1_{\Omega_{1}})$ (see \cite[Theorem~2.1]{DS} and \cite[Lemmas~2.1 and 3.3]{Oeda2011}).

In the case $\mu<0$, we investigate a sufficient region where \eqref{sp} 
admits no positive solution. 
To this end, we introduce the following negative function:
\[
m(\lambda,\alpha)=
\begin{cases}
-c\lambda & \text{for } 0<\lambda\le\dfrac{c}{\alpha},\\[6pt]
-\dfrac{\alpha}{4}\left(\lambda+\dfrac{c}{\alpha}\right)^{2} 
& \text{for } \lambda>\dfrac{c}{\alpha}.
\end{cases}
\]
For any fixed $\alpha>0$, 
$m(\lambda,\alpha)$ 
is monotone decreasing with respect to $\lambda>0$. 
In particular, at the junction point $\lambda=c/\alpha$ 
where the straight line and the parabola meet, 
$m$ is continuously differentiable. 
Furthermore,
\[
m(\,\cdot\,,\alpha)\searrow -\infty 
\quad \text{as } \alpha\nearrow\infty 
\ \text{uniformly on any compact subset of } (0,\infty).
\]

\begin{prop}\label{nonexprop2}
If $\mu\le m(\lambda, \alpha )$, then
\eqref{sp2}, equivalently \eqref{sp}, admits no
positive solution.
\end{prop}

\begin{proof}
Let $(u(x), v(x))$ be any positive solution of \eqref{sp2}.
Integrating the second equation of \eqref{sp2} over $\Omega_1$, we have
\begin{equation}\label{int3}
\int_{\Omega_1}\frac{v(x)}{1+\alpha u(x)}\,
h(u(x),v(x))\,
dx=0,
\end{equation}
where
\[
h(u,v)
:=\alpha u(\lambda -u-bv)+\mu +cu-v.
\] 
By virtue of \eqref{pfape1} and
\eqref{pfape3}, we investigate
the maximum of $h(u,v)$
on the rectangle
\[\mathcal{R}:=\{\,(u,v)\,:\,
0\le u\le \lambda,\quad
0\le v\le\max\{\lambda/b, \mu+c\lambda\}\,\}.\]
Since $h(u,v)$ is monotone decreasing with respect to
$v\in [0, \max\{\lambda/b, \mu+c\lambda\}]$
for any fixed $u\in [0,\lambda ]$,
then
\[
\max_{\mathcal{R}}h(u,v)=
\max_{0\le u\le\lambda}h(u,0)=\max_{0\le u\le\lambda}
\{\,\alpha u(\lambda -u)+\mu+cu\,\}.\]
Therefore, it is easy to verify
\[
\max_{\mathcal{R}}h(u,v)=\begin{cases}
h(\lambda, 0)=\mu+c\lambda\quad&\mbox{if } 0<\lambda\le\dfrac{c}{\alpha},\\
h\left(\dfrac{\lambda}{2}+\dfrac{c}{2\alpha},0\right)=
\mu+\dfrac{\alpha}{4}\left(\lambda+\dfrac{c}{\alpha}\right)^{2}
\quad&\mbox{if } \lambda>\dfrac{c}{\alpha}.
\end{cases}
\]
Suppose that $\mu\le m(\lambda, \alpha )$.
Then it follows that
\[
h(u(x),v(x))< h(u(x),0)\le 
\max_{\mathcal{R}} h(u,v)\le 0\qquad \text{for all }x\in\Omega_1.
\]
This inequality contradicts \eqref{int3}.
Therefore, there is no positive solution for \eqref{sp2} when $\mu\le m(\lambda,\alpha)$.
The proof of
Proposition~\ref{nonexprop2} is complete.
\end{proof}
\subsection{Bifurcation structure of positive stationary solutions}
In the following lemmas, we take $\lambda$ as a bifurcation parameter. 
We will apply the local bifurcation theorem of Crandall and Rabinowitz 
\cite[Theorem 1.7]{CR1} 
to \eqref{sp2} in order to obtain a branch of positive solutions 
which bifurcates from $\varGamma_u$ or $\varGamma_v$. 
Here $\varGamma_u$ and $\varGamma_v$ are sets of semitrivial solutions 
defined by \eqref{gammauv}. Moreover, 
for $p>N$, $X$ is the Sobolev space defined by \eqref{w2pn} and
$$
Y:=L^p (\Omega )\times L^p (\Omega_1 ).
$$
As a beginning of the bifurcation analysis for \eqref{sp2},
we investigate the local bifurcation from $\varGamma_v$ for any fixed $\mu >0$. 
\begin{lem}\label{lb1}
Assume that $\mu > 0$. Positive solutions of \eqref{sp2} bifurcate from $\varGamma_v$ 
if and only if $\lambda =\sigma_1 (b\mu\1_{\Omega_1} )$
in the sense that
all positive solutions of \eqref{sp2} 
near $(\sigma_1 (b\mu\1_{\Omega_1} ),0,\mu )\in \mathbb{R}\times X$ 
can be expressed as
$$
\varGamma_{1, \delta} =\left\{ (\lambda ,u,v)=
\left( \lambda (s),s(\phi^* +\widetilde{u}(s) ),
\mu +s(\psi^* +\widetilde{v}(s))\right) :s\in (0,\delta )\right\}
$$
for some $\delta >0$. Here $(\lambda (s),\widetilde{u}(s), \widetilde{v}(s))$ are smooth functions 
with respect to $s$ 
and satisfy $(\lambda (0),\widetilde{u}(0), \widetilde{v}(0))=(\sigma_1 (b\mu\1_{\Omega_1} ),0,0)$,
$\int_{\Omega}\widetilde{u}(s)\phi^* dx=0$ for any $s\in (0,\delta)$, and
\begin{equation}\label{direction}
\lambda^{\prime}(0)
=\int_{\Omega}(\phi^* +b\1_{\Omega_1} \psi^* )\left( \phi^* \right)^2 dx.
\end{equation}
Moreover, there exists a positive constant $\alpha^{*}$ such that
\begin{equation}\label{direction2}
\lambda'(0)
    \begin{cases}
        >0\quad&\mbox{for}\quad 0<\alpha<\alpha^{*},\\
=0\quad&\mbox{for}\quad\alpha=\alpha^{*},\\
<0\quad&\mbox{for}\quad\alpha>\alpha^{*}.
    \end{cases}
\end{equation}
\end{lem}
\begin{proof}
The nonlinear map 
$
F:\mathbb{R}\times X \to Y
$
associated with \eqref{sp2} is defined by
\[
F(\lambda,u,v)=
\begin{bmatrix}
\Delta u + u(\lambda - u - b\,\1_{\Omega_1} v) \\[6pt]
\Delta v + \dfrac{v}{1+\alpha u}\bigl\{ \alpha u(\lambda - u - b v) + \mu + c u - v \bigr\}
\end{bmatrix}.
\]
Hence $F(\lambda ,u,v)=(0,0)$ if and only if $(u,v)$ is a solution of \eqref{sp2}. 
In particular, $F(\lambda ,0,\mu )=(0,0)$ for any $\lambda$. 
Elementary calculations show that 
the Fr\'{e}chet derivative of $F$ at $(u,v)=(0,\mu )$ is
\begin{equation}\label{fuv}
F_{(u,v)}(\lambda ,0,\mu )\left[
\begin{array}{l}
\phi \\
\psi
\end{array}
\right]
=\left[
\begin{array}{l}
\Delta \phi +(\lambda -b\mu\1_{\Omega_1} )\phi \\
\Delta \psi -\mu \psi +\mu \{ \alpha (\lambda -b\mu )+c\} \phi
\end{array}
\right].
\end{equation}
By the Krein-Rutman theorem, $F_{(u,v)}(\lambda ,0,\mu )[\phi ,\psi ]=0$ 
has a solution with $\phi >0$ if and only if $\lambda =\sigma_1 (b\mu\1_{\Omega_1} )$. 
Hence $\lambda =\sigma_1 (b\mu\1_{\Omega_1} )$ is the only possible bifurcation point where 
positive solutions of \eqref{sp2} bifurcate from $\varGamma_v$. 
From \eqref{phistar1}, \eqref{psistar1} and \eqref{fuv}, the kernel of $F_{(u,v)}(\sigma_1 (b\mu\1_{\Omega_1} ),0,\mu )$ 
is given by 
$$
\mbox{Ker}\, F_{(u,v)}(\sigma_1 (b\mu\1_{\Omega_1} ),0,\mu )=\mbox{Span}\{ (\phi^* ,\psi^* )\}
$$ 
and thus $\mbox{dim}\, \mbox{Ker}\, F_{(u,v)}(\sigma_1 (b\mu\1_{\Omega_1} ),0,\mu )=1$. 
Moreover, due to the Fredholm alternative theorem, 
the range of $F_{(u,v)}(\sigma_1 (b\mu\1_{\Omega_1} ),0,\mu )$ is given by 
\begin{equation}\label{range1}
\mbox{Ran}\, F_{(u,v)}(\sigma_1 (b\mu\1_{\Omega_1} ),0,\mu )=
\left\{ (\phi ,\psi )\in Y:\int_{\Omega}\phi \phi^* dx=0\right\}
\end{equation}
and hence $\mbox{codim}\, \mbox{Ran}\, F_{(u,v)}(\sigma_1 (b\mu\1_{\Omega_1} ),0,\mu )=1$. 
Furthermore, \eqref{range1} yields
\begin{equation}\label{Fluv}
    F_{\lambda (u,v)}(\sigma_1 (b\mu\1_{\Omega_1} ),0,\mu )\left[
\begin{array}{c}
\phi^* \\
\psi^*
\end{array}
\right]
=\left[
\begin{array}{c}
\phi^* \\
\alpha \mu \phi^*
\end{array}
\right]
\not\in \mbox{Ran}\, F_{(u,v)}(\sigma_1 (b\mu\1_{\Omega_1} ),0,\mu )
\end{equation}
because of $\phi^* >0$. 
Therefore, we can apply the local bifurcation theorem \cite[Theorem 1.7]{CR1} to $F$ 
at $(\sigma_1 (b\mu\1_{\Omega_1} ),0,\mu )$. 

In order to show \eqref{direction} and \eqref{direction2}, we recall the following direction formula at the bifurcation point (see e.g., \cite[(I.6.3)]{K}):
$$
\lambda^{\prime}(0)=
-\frac{\langle F_{(u,v)(u,v)}(\sigma_1 (b\mu\1_{\Omega_1} ),0,\mu )[\phi^* ,\psi^* ]^2 ,\ell_1 \rangle}
{2\langle F_{\lambda (u,v)}(\sigma_1 (b\mu\1_{\Omega_1} ),0,\mu )[\phi^* ,\psi^* ],\ell_1 \rangle},
$$
where $\langle\,\cdot\,,\,\cdot\,\rangle$ denotes the duality between $Y$ and $Y^{*}$.
The functional
$\ell_1\in Y^{*}$ is defined by
$\langle [\phi, \psi], \ell_{1}\rangle=\int_{\Omega}\phi\phi^{*}dx$.
Then for $f(u,v)=u(\lambda -u-b\1_{\Omega_{1}}v)$, it follows that
\begin{equation}
\begin{split}
&\langle F_{(u,v)(u,v)}(\sigma_1 (b\mu\1_{\Omega_1} ),0,\mu )[\phi^* ,\psi^* ]^2 ,\ell_1 \rangle \\
=&
\int_{\Omega}[\phi^{*}, \psi^{*}]
\left[
\begin{array} {cc}
f_{uu} & f_{uv} \\
f_{vu} & f_{vv}
\end{array}
\right]\bigg|_{(\lambda, u, v)=(\sigma_1 (b\mu\1_{\Omega_1} ),0,\mu )}
\left[
\begin{array}{c}
\phi^{*}\\
\psi^{*} 
\end{array}
\right]
\phi^{*}\,dx \\
=&
\int_{\Omega}[\phi^{*}, \psi^{*}]
\left[
\begin{array} {cc}
-2 & -b\1_{\Omega_{1}} \\
-b\1_{\Omega_{1}} & 0
\end{array}
\right]
\left[
\begin{array}{c}
\phi^{*}\\
\psi^{*} 
\end{array}
\right]
\phi^{*}\,dx\\
=&
-2\int_{\Omega}(\phi^{*}+b\1_{\Omega_1}\psi^{*})(\phi^{*})^2\,dx.
\end{split}
\nonumber
\end{equation}
Together with $\langle F_{\lambda (u,v)}(\sigma_1 (b\mu\1_{\Omega_1} ),0,\mu )[\phi^* ,\psi^* ],\ell_1 \rangle =
\|\phi^{*}\|_{L^2(\Omega)}^{2}=1$ from \eqref{phistar1} and \eqref{Fluv}, we obtain \eqref{direction}. In view of \eqref{psistar1}, we recall that
$\psi^{*}$ depends on $\alpha $ and is given by
\begin{equation}\label{phistar1sai}
    \psi^* (\alpha )=(-\Delta +\mu )^{-1}_{\Omega_1}
\left( 
  \mu \left[ 
    \alpha \left\{ \sigma_1(b\mu\1_{\Omega_1}) - b\mu \right\} + c
  \right] \phi^*
\right).
\end{equation}
Since $\sigma_{1}(b\mu)=b\mu$ and since $\sigma_{1}(q)$ is monotone increasing 
with respect to $q\in L^{\infty}(\Omega)$, we obtain
\[
\sigma_{1}(b\mu\1_{\Omega_{1}})-b\mu
= \sigma_{1}(b\mu\1_{\Omega_{1}})-\sigma_{1}(b\mu)
<0.
\]
Indeed, this follows from the fact that $b\mu\1_{\Omega_{1}}\le b\mu$ in $\Omega$ 
with strict inequality in $\Omega_{0}$.
Taking into account that $\phi^{*}>0$ is independent of $\alpha$ by \eqref{phistar1},
we deduce from \eqref{phistar1sai} that $\psi^{*}(0)>0$ in $\Omega_{1}$, 
that $\psi^{*}(\alpha)$ is monotone decreasing on $\Omega_{1}$ with respect to $\alpha\ge0$, 
and that $\psi^{*}(\alpha)\to -\infty$ uniformly on $\overline{\Omega}_{1}$ as $\alpha\to\infty$.
By \eqref{direction}, we can find $\alpha^*>0$ such that \eqref{direction2} holds. Thus the proof of Lemma \ref{lb1} is complete.
\end{proof}

\begin{rem}
In the case $\mu=0$, the linearized operator satisfies
\[
\mbox{Ker} F_{(u,v)}(0,0,0)=\mbox{Span}\{(1,0),(0,1)\},
\]
and hence $\mbox{Ker} F_{(u,v)}(0,0,0)$ is two-dimensional.
Therefore, the bifurcation theorem from simple eigenvalues
by Crandall-Rabinowitz \cite[Theorem~1.7]{CR1} is not applicable.
Nevertheless, by applying a standard Lyapunov-Schmidt reduction, one can
analyze the bifurcation structure at $\lambda=0$ by decomposing the phase space
into the two-dimensional kernel and its complementary subspace.
This procedure shows that, in addition to the semi-trivial branches, there also
exists a branch of positive solutions emanating from $\lambda=0$, which can be
parameterized in the form $\varGamma_{1,\delta}$.
\end{rem}
Next we investigate the local bifurcation of positive solutions from 
$\varGamma_{u}$ for any fixed $\mu<0$.
\begin{lem}\label{lb2}
Assume that $\mu <0$. Positive solutions of \eqref{sp2} bifurcate from $\varGamma_u$ 
if and only if $\lambda =-\mu /c$. 
In addition, all positive solutions of \eqref{sp2} 
near $(-\mu /c,\lambda ,0)\in \mathbb{R}\times X$ 
can be expressed as
$$
\varGamma_{2, \delta} =\left\{ (\lambda ,u,v)=
\left( \lambda (s),\lambda +s(\phi_* +\widetilde{u}(s)),s(1+\widetilde{v}(s))\right) :s\in (0,\delta )\right\}
$$
for some $\delta >0$. Here $(\lambda (s),\widetilde{u}(s), \widetilde{v}(s))$ are 
smooth functions with respect to $s$ 
and satisfy $(\lambda (0),\widetilde{u}(0), \widetilde{v}(0))=(-\mu /c,0,0)$, $\int_{\Omega_1}\widetilde{v}(s)dx=0$ and
\[\lambda'(0)=\dfrac{c^{2}\|\phi_{*}\|_{L^1(\Omega_{1})}+c|\Omega_{1}|+
\alpha |\mu|(b\,|\Omega_1|-\|\phi_{*}\|_{L^{1}(\Omega_1)})}{c^{2}\,|\Omega_{1}|},\]
which is positive for any $\alpha\ge 0$.
\end{lem}
\begin{proof}
Let $U=u-\lambda$ in \eqref{sp2} and define a mapping $G:\mathbb{R}\times X\to Y$ by
$$
G(\lambda ,U,v)=\left[
\begin{array}{l}
\Delta U -(U+\lambda )(U+b\1_{\Omega_1}v)\\
\Delta v +\dfrac{v}{1+\alpha (U+\lambda )}
\left\{ -\alpha (U+\lambda )(U+bv)+\mu +c(U+\lambda )-v\right\}
\end{array}
\right].
$$
Then $G(\lambda ,U,v)=(0,0)$ if and only if $(U+\lambda ,v)$ is a solution of \eqref{sp2}. 
In particular, $G(\lambda ,0 ,0)=(0,0)$ for any $\lambda$. 
Simple calculations yield
$$
G_{(U,v)}(\lambda ,0,0)\left[
\begin{array}{l}
\phi \\
\psi
\end{array}
\right]
=\left[
\begin{array}{l}
\Delta \phi -\lambda \phi -b\1_{\Omega_1}\lambda \psi \\
\Delta \psi +\dfrac{\mu +c\lambda}{1+\alpha \lambda}\psi
\end{array}
\right].
$$
It can be verified that $\lambda =-\mu /c$ is the only possible bifurcation point where 
positive solutions of \eqref{sp2} bifurcate from $\varGamma_u$. 
Moreover, we  find 
$$
\mbox{Ker}\, G_{(U,v)}(-\mu /c,0,0)=\mbox{Span}\{ (\phi_* ,1)\}
$$
and
\begin{equation}\label{range2}
\mbox{Ran}\, G_{(U,v)}(-\mu /c,0,0)=
\left\{ (\phi ,\psi )\in Y:\int_{\Omega_1}\psi dx=0\right\} .
\end{equation}
Thus $\mbox{dim}\, \mbox{Ker}\, G_{(U,v)}(-\mu /c,0,0)
=\mbox{codim}\, \mbox{Ran}\, G_{(U,v)}(-\mu /c,0,0)=1$. 
Furthermore, \eqref{range2} yields
\begin{equation}\label{GlUv}
G_{\lambda (U,v)}(-\mu /c,0,0)\left[
\begin{array}{l}
\phi_* \\
1
\end{array}
\right]
=\left[
\begin{array}{l}
-\phi_* -b\1_{\Omega_1}\\
\frac{c^2}{c-\alpha \mu}
\end{array}
\right]
\not\in \mbox{Ran}\, G_{(U,v)}(-\mu /c,0,0).
\end{equation}
Therefore, we can apply the local bifurcation theorem \cite{CR1} to $G$ 
at $(-\mu /c,0,0)$. 
By the direction formula 
at the bifurcation point (e.g., \cite[(I.6.3)]{K}), one can see
\begin{align*}
\lambda^{\prime}(0)
=-\frac{\langle G_{(U,v)(U,v)}(-\mu /c,0,0)[\phi_* ,1]^2 ,\ell_2 \rangle}
{2\langle G_{\lambda (U,v)}(-\mu /c,0,0)[\phi_* ,1],\ell_2 \rangle},
\end{align*}
where the functional $\ell_{2}\in Y^{*}$ is defined by
$\langle[\phi, \psi],\ell_{2}\rangle =\int_{\Omega_1}\psi\,dx$.
Then for 
\[g(U,v):=\dfrac{v}{1+\alpha (U+\lambda)}\{ -\alpha (U+\lambda )(U+bv)+\mu +c(U+\lambda )-v\},\]
straightforward calculation gives
\begin{equation}
\begin{split}
&\langle G_{(U,v)(U,v)}(-\mu/c, 0, 0) )[\phi_*, 1]^2 ,\ell_2 \rangle \\
=&
\int_{\Omega_1}[\phi_{*}, 1]
\left[
\begin{array} {cc}
g_{UU} & g_{Uv} \\
g_{vU} & g_{vv}
\end{array}
\right]\bigg|_{(\lambda, u, v)=(-\mu/c,0, 0 )}
\left[
\begin{array}{c}
\phi_{*}\\
1
\end{array}
\right]
\,dx \\
=&
\int_{\Omega_1}[\phi_{*}, 1]
\left[
\begin{array} {cc}
0 & \frac{c^{2}+\alpha\mu}{c-\alpha\mu} \\
\frac{c^{2}+\alpha\mu}{c-\alpha\mu} & -\frac{2(c-\alpha b\mu)}{c-\alpha\mu}
\end{array}
\right]
\left[
\begin{array}{c}
\phi_{*}\\
1 
\end{array}
\right]
\,dx\\
=&
\dfrac{2}{c-\alpha\mu}\bigg\{
(c^2
+\alpha\mu )\int_{\Omega_1}\phi_{*}\,dx +(-c+\alpha b\mu  )|\Omega_{1}|\biggr\}\\
=&
\dfrac{-2}{c-\alpha \mu}\left\{
c^{2}\|\phi_{*}\|_{L^1(\Omega_{1})}+c|\Omega_{1}|+
\alpha |\mu|(b\,|\Omega_1|-\|\phi_{*}\|_{L^{1}(\Omega_1)})\right\},
\end{split}
\nonumber
\end{equation}
where the last expression is presented taking into account 
the negativity of $\phi_{*}=(-\Delta +\lambda)^{-1}_{\Omega}[-b\1_{\Omega_{1}}\lambda ]$ and $\mu$.
By \eqref{GlUv} we see
\[\langle G_{\lambda (U,v)}(-\mu/c, 0, 0)[\phi_{*}, 1], \ell_{2}\rangle=\dfrac{c^{2}}{c-\alpha\mu}|\Omega_{1}|.\]
Therefore, we obtain 
\[\lambda'(0)=\dfrac{c^{2}\|\phi_{*}\|_{L^1(\Omega_{1})}+c|\Omega_{1}|+
\alpha |\mu|(b\,|\Omega_1|-\|\phi_{*}\|_{L^{1}(\Omega_1)})}{c^{2}\,|\Omega_{1}|}.\]
In order to show that $\lambda'(0)>0$ for any $\alpha\ge 0$,
we recall the equation of $\phi_{*}$:
\[-\Delta\phi_{*}+\lambda\phi_{*}=-b\1_{\Omega_{1}}\lambda
\quad\mbox{in}\ \Omega,\qquad \partial_{n}\phi_{*}=0\quad\mbox{on}\ \partial\Omega.\]
Integrating this elliptic equation over $\Omega$, we see
$\|\phi_{*}\|_{L^{1}(\Omega )}=b|\Omega_{1}|$ thanks to the boundary condition.
Since $\Omega_{1}\subsetneq\Omega$, we can see
\[\|\phi_{*}\|_{L^{1}(\Omega_{1})}<\|\phi_{*}\|_{L^{1}(\Omega )}=b|\Omega_{1}|.\]
Therefore, we find that the bifurcation at $(\lambda, U, v)=(-\mu/c, 0,0)$ is supercritical in the sense of $\lambda'(0)>0$ for any $\alpha\ge 0$.
Then we complete the proof of Lemma \ref{lb2}.
\end{proof}
We are now in a position to complete the proofs of Theorems \ref{exthm1} and \ref{exthm2}. 
We will prove Theorem \ref{exthm1} only 
since the proof of Theorem \ref{exthm2} is similar to that of Theorem \ref{exthm1}.
\begin{proof}[Proof of Theorem \ref{exthm1}]
We first discuss the case $\mu >0$. 
For the local bifurcation branch $\varGamma_{1,\delta}$ obtained in Lemma \ref{lb1}, 
Let $\varGamma_1 \subset \mathbb{R}\times X$ 
denote the maximal connected set satisfying
\begin{equation}\label{pfexthm11}
\begin{split}
\varGamma_{1,\delta}&\subset \varGamma_1\\
&\subset 
\{(\lambda ,u,v)\in (\mathbb{R}\times X)\setminus \{ (\sigma_1 (b\mu\1_{\Omega_1} ),0,\mu )\} 
:(\lambda ,u,v)\mbox{ is a solution of \eqref{sp2}}\} .
\end{split}
\end{equation}
Define 
$P:=\{ (u,v)\in X:u>0\ \mbox{in} \ \overline{\Omega},\ v>0 \ \mbox{in} \ \overline{\Omega}_1 \}$. 
We will prove
\begin{equation}\label{pfexthm12}
\varGamma_1 \subset \mathbb{R}\times P
\end{equation}
by contradiction. Suppose that $\varGamma_1 \not\subset \mathbb{R}\times P$. 
Then we can take a sequence 
$\{ (\lambda_k ,u_k ,v_k )\}_{k=1}^{\infty}\subset \varGamma_1 \cap (\mathbb{R}\times P)$ 
and
\begin{equation}\label{pfexthm13}
(\lambda_{\infty},u_{\infty},v_{\infty}) \in \varGamma_1 \cap (\mathbb{R}\times \partial P)
\end{equation}
such that
$$
\lim_{k\to \infty}(\lambda_k ,u_k ,v_k )=
(\lambda_{\infty},u_{\infty},v_{\infty}) \ \ \mbox{in}\ \mathbb{R}\times X,
$$
where $(u_{\infty},v_{\infty})$ is a nonnegative solution 
of \eqref{sp2} with $\lambda =\lambda_{\infty}$. Due to the strong maximum principle (see e.g., \cite{Bo}), 
one of the following (a)-(c) must occur:
\begin{itemize}
\item[(a)]
$u_{\infty}=0$ in $\overline{\Omega}$,\quad $v_{\infty}=0$ in $\overline{\Omega}_1$.
\item[(b)]
$u_{\infty}>0$ in $\overline{\Omega}$,\quad $v_{\infty}=0$ in $\overline{\Omega}_1$.
\item[(c)]
$u_{\infty}=0$ in $\overline{\Omega}$,\quad $v_{\infty}>0$ in $\overline{\Omega}_1$.
\end{itemize}
Integrating the second equation of \eqref{sp2} with $(\lambda ,u,v)=(\lambda_k ,u_k ,v_k )$, we have
\begin{equation}\label{pfexthm14}
\int_{\Omega_1}\dfrac{v_k}{1+\alpha u_k}
\left\{ \alpha u_k (\lambda_k -u_k -bv_k )+\mu +cu_k -v_k \right\} dx=0
\end{equation}
for any $k\in \mathbb{N}$. If (a) or (b) holds, then the integrand in \eqref{pfexthm14} is 
positive for sufficiently large $k\in \mathbb{N}$ because of 
$0<u_k \le \lambda_k$, $v_k >0$ and $\mu >0$. 
This contradicts \eqref{pfexthm14}. If (c) holds, then
$$
\Delta v_{\infty}+v_{\infty}(\mu -v_{\infty})=0\ \ \mbox{in}\ \Omega_1 ,
\ \ \ \partial_n v_{\infty} =0\ \ \mbox{on}\ \partial \Omega_1 ,
\ \ \ v_{\infty}>0\ \ \mbox{in}\ \overline{\Omega}_1
$$
and thus $v_{\infty}=\mu$ in $\overline{\Omega}_1$. Then 
$(\lambda_{\infty} ,u_{\infty},v_{\infty})=(\sigma_1 (b\mu\1_{\Omega_1} ),0,\mu )$ must hold by Lemma \ref{lb1}. 
But this contradicts \eqref{pfexthm11} and \eqref{pfexthm13}. Hence 
the assertion \eqref{pfexthm12} holds true. 
Thanks to Rabinowitz's global bifurcation theorem \cite{R},
together with its refinement to the unilateral/global alternatives
for $C^1$ Fredholm maps developed in
L{\'o}pez-G{\'o}mez and Mora-Corral \cite{LGMC},
L{\'o}pez-G{\'o}mez \cite{LG16}, and
L{\'o}pez-G{\'o}mez and Sampedro \cite{LGSam}
(see also Shi and Wang \cite{SW}),
$\varGamma_1$ satisfies one of the followings:
\begin{itemize}
\item[(i)]
$\varGamma_1$ is unbounded in $\mathbb{R}\times X$.
\item[(ii)]
$\varGamma_1$ contains a point $(\bar{\lambda},0,\mu )$ 
with $\bar{\lambda}\neq \sigma_1 (b\mu\1_{\Omega_1} )$.
\item[(iii)]
$\varGamma_1$ contains a point $(\tilde{\lambda}, \tilde{\phi}, \tilde{\psi})\in \mathbb{R}\times X$ 
with $(\tilde{\phi}, \tilde{\psi})\neq (0,\mu )$ and $\int_{\Omega}\tilde{\phi}\phi^* dx=0$, 
where $\phi^*$ is the function given by \eqref{phistar1}.
\end{itemize}
On account of \eqref{pfexthm12}, case (ii) cannot occur. 
Case (iii) is also impossible by \eqref{pfexthm12} and $\phi^* >0$. Therefore, case (i) must hold. 
Consequently, the conclusion for the case $\mu >0$ follows from 
\eqref{pfexthm12}, Lemmas \ref{ape} and \ref{lb1} and Proposition \ref{nonexprop1}.

Next we consider the case $\mu =0$. Fix any $\lambda >0$. 
By virtue of the result for the case $\mu >0$, we can find a sequence 
$\{ (\mu_k , u_k ,v_k )\}_{k=1}^{\infty}$ such that $(u_k ,v_k )$ is a positive solution of \eqref{sp2} 
with $\mu =\mu_k$ and $\lim_{k\to \infty}\mu_k =0$. 
Then Lemma \ref{ape} implies that
$$
\lim_{k\to \infty}(u_k ,v_k )=(u_{\infty},v_{\infty}) \ \ \mbox{in}\ 
C^1 (\overline{\Omega})\times C^1 (\overline{\Omega}_1 )
$$
by passing to a subsequence if necessary, where 
$(u_{\infty},v_{\infty})$ is a nonnegative solution of \eqref{sp2} with $\mu =0$. 
By the strong maximum principle, $(u_{\infty},v_{\infty})$ satisfies either 
$u_{\infty}>0$ in $\overline{\Omega}$ and $v_{\infty}>0$ in $\overline{\Omega}_1$ or 
one of (a)-(c) which appeared in the proof for the case $\mu >0$. 
From \eqref{sp2} with $(\mu ,u,v)=(\mu_k ,u_k ,v_k )$, we have
\begin{equation}\label{pfexthm15}
\int_{\Omega}u_k (\lambda -u_k -b\1_{\Omega_1}v_k )dx=0
\end{equation}
and
\begin{equation}\label{pfexthm16}
\int_{\Omega_1}\dfrac{v_k}{1+\alpha u_k}
\left\{ \alpha u_k (\lambda -u_k -bv_k )+\mu_k +cu_k -v_k \right\} dx=0
\end{equation}
for any $k\in \mathbb{N}$. 
Case (a) cannot occur by \eqref{pfexthm15} and $\lambda >0$, 
and case (b) cannot happen by 
\eqref{pfexthm16}, $u_k \le \lambda$ and $\lim_{k\to \infty}\mu_k =0$. 
Case (c) is also impossible by \eqref{pfexthm16} and $\lim_{k\to \infty}\mu_k =0$. 
Thus $u_{\infty}>0$ in $\overline{\Omega}$ and $v_{\infty}>0$ in $\overline{\Omega}_1$ 
must hold. This means the existence of a positive solution of \eqref{sp2} 
with $\mu =0$ for any fixed $\lambda >0$. 
Hence we have completed the proof of Theorem \ref{exthm1}.
\end{proof}
\section{Proof of Theorem \ref{abthm}}
In this section, We will prove Theorem \ref{abthm}. 
We first show the following lemma.
\begin{lem}\label{abthm-l1}
The problem
\begin{equation}\label{LP1}
\begin{cases}
\Delta u+u(\lambda -u-b\1_{\Omega_1}v)=0\ \ &\mbox{in}\ \Omega,\\
\Delta v+v(\lambda -u-bv)=0
\ \ &\mbox{in}\ \Omega_1 ,\\
\partial_n u=0\ \
\ \ &\mbox{on}\ \partial \Omega,\\
\partial_n v=0\ \
\ \ &\mbox{on}\ \partial \Omega_1
\end{cases}
\end{equation}
has no positive solution.
\end{lem}
\begin{proof}
Let $(u,v)$ be any nonnegative solution of \eqref{LP1} with $u\not\equiv 0$. 
Then $u>0$ in $\overline{\Omega}$ by the strong maximum principle. 
Multiplying the first equation of \eqref{LP1} by $(\lambda /u)-1$ 
and integrating the resulting equation over $\Omega$, we have
\begin{align}\label{LP1-1}
0&=\lambda \left( \int_{\Omega}\dfrac{|\nabla u|^2}{u^2}dx
+\int_{\Omega}(\lambda -u-b\1_{\Omega_1}v)\,dx \right)
-\int_{\Omega}u(\lambda -u-b\1_{\Omega_1}v)\,dx\nonumber\\
&=\lambda \int_{\Omega}\dfrac{|\nabla u|^2}{u^2}dx
+\int_{\Omega}(\lambda -u)^2 dx
-b\int_{\Omega_1}v(\lambda -u)\,dx.
\end{align}
On the other hand,
\begin{equation}\label{LP1-2}
\begin{split}
b^2 \int_{\Omega_1}v^2 dx-b\int_{\Omega_1}v(\lambda -u)\,dx
&=-b\int_{\Omega_1}v(\lambda -u-bv)\,dx\\
&=b\int_{\Omega_1}\Delta v\,dx =b\int_{\partial\Omega_1}\partial_n v\,dS =0.
\end{split}
\end{equation}
Hence we see from \eqref{LP1-1} and \eqref{LP1-2} that
\begin{align*}
0&=\lambda \int_{\Omega}\dfrac{|\nabla u|^2}{u^2}dx
+\int_{\Omega}(\lambda -u)^2 dx
+b^2 \int_{\Omega_1}v^2 dx-2b\int_{\Omega_1}v(\lambda -u)\,dx\\
&\ge \lambda \int_{\Omega}\dfrac{|\nabla u|^2}{u^2}dx
+\int_{\Omega}(\lambda -u)^2 dx
-\int_{\Omega_1}(\lambda -u)^2 dx\\
&=\lambda \int_{\Omega}\dfrac{|\nabla u|^2}{u^2}dx
+\int_{\Omega_0}(\lambda -u)^2 dx,
\end{align*}
where we have used the fact that
$$
b^2 v^2 -2bv(\lambda -u)
=\{ bv-(\lambda -u)\}^2 -(\lambda -u)^2 \ge -(\lambda -u)^2 .
$$
Therefore, $u\equiv \lambda$ in $\overline{\Omega}$ must hold. 
This implies that $(u,v)=(\lambda ,0)$ 
since $(u,v)$ satisfies \eqref{LP1}. 
Thus the proof is complete.
\end{proof}
Next we prove the following a priori estimates.
\begin{lem}\label{abthm-l2}
Let $(w,v)$ be any positive solution of \eqref{LP2}.
Then it holds that
\[\mu>0,\quad \lambda<\sigma_{1}(b\mu\1_{\Omega_{1}})\quad\mbox{and}\quad
\dfrac{\lambda}{b}<v(x)<\mu\quad\mbox{for all}\quad x\in \overline{\Omega}_1.\]
\end{lem}

\begin{proof}
Let $v(x_0 )=\max_{\overline{\Omega}_1}v$ and $v(y_0 )=\min_{\overline{\Omega}_1}v$ 
with $x_0 ,y_0 \in \overline{\Omega}_1$. 
Applying the maximum principle to the second equation of \eqref{LP2}, we have
\begin{equation}\label{LP2-1}
\dfrac{v(x_0 )}{1+w(x_0 )}\left\{ w(x_0 )\left( \lambda -b\max_{\overline{\Omega}_1}v\right)
+\mu -\max_{\overline{\Omega}_1}v\right\} \ge 0
\end{equation}
and
\begin{equation}\label{LP2-2}
\dfrac{v(y_0 )}{1+w(y_0 )}\left\{ w(y_0 )\left( \lambda -b\min_{\overline{\Omega}_1}v\right)
+\mu -\min_{\overline{\Omega}_1}v\right\} \le 0.
\end{equation}
Then we see from \eqref{qrs}, \eqref{LP2-1} and \eqref{LP2-2} that
\begin{equation}\label{vest}
\min \left\{ \dfrac{\lambda}{b},\, \mu \right\}
\le \dfrac{\lambda w(y_0 )+\mu}{bw(y_0 )+1}
\le \min_{\overline{\Omega}_1}v
\le \max_{\overline{\Omega}_1}v
\le \dfrac{\lambda w(x_0 )+\mu}{bw(x_0 )+1}
\le \max \left\{ \dfrac{\lambda}{b},\, \mu \right\} .
\end{equation}
In particular,
\begin{equation}\label{2vest2}
\min \left\{ \dfrac{\lambda}{b},\, \mu \right\}
<\min_{\overline{\Omega}_1}v
\le \max_{\overline{\Omega}_1}v
<\max \left\{ \dfrac{\lambda}{b},\, \mu \right\}
\end{equation}
if $\lambda \neq b\mu$. 
From the first equation of \eqref{LP2}, 
$\lambda =\sigma_1 (b\1_{\Omega_1}v)$ holds. 
If $\lambda <b\mu$, then $\mu >(\lambda /b)>0$ by assumption, and
$$
\dfrac{\lambda}{b}<v(x)<\mu\quad\mbox{for all}\quad x\in \overline{\Omega}_1 ,\quad \lambda =\sigma_1 (b\1_{\Omega_1}v)
<\sigma_1 (b\mu\1_{\Omega_1} )
$$
by \eqref{2vest2}. If $\lambda >b\mu$, then
$$
\lambda =\sigma_1 (b\1_{\Omega_1}v)
<\sigma_1 \left( b\1_{\Omega_1}\cdot \dfrac{\lambda}{b}\right)
<\sigma_1 (\lambda )=\lambda
$$
by \eqref{2vest2}. But this is impossible. 
If $\lambda =b\mu$, then 
\eqref{vest} yields
$$
\mu \le \min_{\overline{\Omega}_1}v\le \max_{\overline{\Omega}_1}v\le \mu .
$$
This means that $v\equiv \mu$ in $\overline{\Omega}_1$ and thus
$$
\lambda =\sigma_1 (b\1_{\Omega_1}v)
=\sigma_1 (b\mu\1_{\Omega_1} )
<\sigma_1 (b\mu )
=b\mu =\lambda .
$$
But this is also impossible. 
Hence the proof is complete.
\end{proof}
By virtue of Lemmas \ref{ape}, \ref{abthm-l1} and \ref{abthm-l2}, 
the following lemma holds true.
\begin{lem}\label{abthm-l4}
Let $\{ (u_k ,v_k )\}_{k=1}^{\infty}$ be any sequence such that 
$(u_k ,v_k )$ is a positive solution of \eqref{sp2} with $\alpha =\alpha_k$ 
and $\lim_{k\to \infty}\alpha_k =\infty$.
\begin{enumerate}[{\rm (i)}]
\item
If $\{ \alpha_k \| u_k \|_{L^{\infty}(\Omega )}\}_{k=1}^{\infty}$ is unbounded, then
$$
\lim_{k\to \infty}(u_k ,v_k )=(\lambda ,0)\ 
\mbox{in}\ C^1 (\overline{\Omega})\times C^1 (\overline{\Omega}_1 ).
$$
\item
If $\{ \alpha_k \| u_k \|_{L^{\infty}(\Omega )}\}_{k=1}^{\infty}$ is bounded, then 
$\mu >0$ and
$$
\lim_{k \to \infty}(\alpha_k u_k ,v_k )=(w_{\infty},v_{\infty})\ 
\mbox{in}\ C^1 (\overline{\Omega})\times C^1 (\overline{\Omega}_1 )
$$
by passing to a subsequence if necessary. 
Here $(w_{\infty},v_{\infty})$ satisfies \eqref{LP2} and either
$$
{\rm (A)}\ \ w_{\infty}>0 \ \mbox{in} \ \overline{\Omega},\ 
v_{\infty}>0 \ \mbox{in} \ \overline{\Omega}_1 ,\ \lambda <\sigma_1 (b\mu\1_{\Omega_1} )
$$
or
$$
{\rm (B)}\ \ (w_{\infty},v_{\infty})=(0,\mu ),\ \lambda =\sigma_1 (b\mu\1_{\Omega_1} )
$$
holds.
\end{enumerate}
\end{lem}
\begin{proof}
We first prove part (i). It follows from Lemma \ref{ape} that
$$
\lim_{k \to \infty}(u_k ,v_k )=(u_{\infty},v_{\infty})\ 
\mbox{in}\ C^1 (\overline{\Omega})\times C^1 (\overline{\Omega}_1 )
$$
for a pair of nonnegative functions $(u_{\infty},v_{\infty})\in 
C^1 (\overline{\Omega})\times C^1 (\overline{\Omega}_1 )$ 
by passing to a subsequence if necessary. 
Since $\alpha_k u_k$ is a positive solution of
$$
\Delta (\alpha_k u_k )+\alpha_k u_k (\lambda -u_k -b\1_{\Omega_1}v_k )=0
\ \ \mbox{in}\ \Omega ,\quad \partial_n (\alpha_k u_k )=0
\ \ \mbox{on}\ \partial \Omega
$$
and $\{ \| \lambda -u_k -b\1_{\Omega_1}v_k \|_{L^{p}(\Omega )}\}_{k=1}^{\infty}$ is bounded 
for any $p>1$ by Lemma \ref{ape}, 
the application of the Harnack inequality leads to
$$
\max_{\overline{\Omega}}\alpha_k u_k 
\le C\min_{\overline{\Omega}}\alpha_k u_k
$$
for some positive constant $C$ independent of $k$. 
In contrast to the standard application of the
Harnack inequality, we emphasize that the associated potential function
\[
\widetilde{K}^{(1)}_{k}(x) := \lambda -u_k -b\1_{\Omega_1}v_k
\]
may exhibit a discontinuous jump of size $bv_{k}(x)$ across 
the interface $\partial\Omega_{0}\,(\subset\Omega)$.
Nevertheless, even in this situation, one can invoke 
the generalized Harnack inequality due to Lou and Ni 
\cite[Lemma~3.1]{LN2}, whose proof relies on the classical 
framework of Stampacchia \cite{Stampacchia1965}, in particular 
on the Harnack inequality established there for divergence-form 
operators with merely bounded measurable coefficients and 
potential functions in $L^{N/2}(\Omega)$.

Thus $\min_{\overline{\Omega}}\alpha_k u_k$ is also unbounded. 
Then letting $k\to \infty$ in \eqref{sp2} 
with $(\alpha ,u,v)=(\alpha_k ,u_k ,v_k )$, we see that 
$(u_{\infty},v_{\infty})$ is a nonnegative solution of \eqref{LP1}. 
By Lemma \ref{abthm-l1} and the strong maximum principle, 
one of the following (a)-(c) must occur:
$$
\mbox{(a)}\ (u_{\infty},v_{\infty})=(0,0),\qquad
\mbox{(b)}\ (u_{\infty},v_{\infty})=(\lambda ,0),\qquad
\mbox{(c)}\ (u_{\infty},v_{\infty})=\left( 0,\dfrac{\lambda}{b}\right) .
$$
Integrating the first equation of \eqref{sp2} with $(u,v)=(u_k ,v_k )$, 
we have
\begin{equation}\label{abthm-l4-1}
\int_{\Omega}u_k (\lambda -u_k -b\1_{\Omega_1}v_k )\,dx=0
\end{equation}
for any $k\in \mathbb{N}$. If (a) holds, then
$$
\int_{\Omega}u_k (\lambda -u_k -b\1_{\Omega_1}v_k )\,dx>0
$$
for sufficiently large $k$. This contradicts \eqref{abthm-l4-1}. 
Assume that (c) holds and 
set $\tilde{u}_k =u_k /\| u_k \|_{L^{\infty}(\Omega )}>0$. 
Then
$$
\Delta \tilde{u}_k +\tilde{u}_k (\lambda -u_k -b\1_{\Omega_1}v_k )=0
\ \ \mbox{in}\ \Omega ,\quad 
\partial_n \tilde{u}_k =0\ \ \mbox{on}\ \partial \Omega ,
\quad \| \tilde{u}_k \|_{L^{\infty}(\Omega )}=1.
$$
Letting $k\to \infty$ in the above equation, 
we find from Lemma \ref{ape} that
$$
\Delta \tilde{u}_{\infty}+
\tilde{u}_{\infty}\left( \lambda -b\1_{\Omega_1}\cdot \frac{\lambda}{b}\right) =0
\ \ \mbox{in}\ \Omega ,\quad 
\partial_n \tilde{u}_{\infty}=0\ \ \mbox{on}\ \partial \Omega ,
\quad \| \tilde{u}_{\infty}\|_{L^{\infty}(\Omega )}=1
$$
for some nonnegative function 
$\tilde{u}_{\infty}\in C^1 (\overline{\Omega})$. 
By $\| \tilde{u}_{\infty}\|_{L^{\infty}(\Omega )}=1$ and the strong maximum principle, 
$\tilde{u}_{\infty}>0$ in $\overline{\Omega}$ holds. 
But this leads to a contradiction:
$$
0=\int_{\Omega}\tilde{u}_{\infty}
\left( \lambda -b\1_{\Omega_1}\cdot \frac{\lambda}{b}\right) dx
=\lambda \int_{\Omega_0}\tilde{u}_{\infty}\,dx>0.
$$
Hence we have proved part (i).

Next we prove part (ii). By Lemma \ref{ape}, there exists 
a pair of nonnegative functions $(w_{\infty},v_{\infty})\in 
C^1 (\overline{\Omega})\times C^1 (\overline{\Omega}_1 )$ 
such that
$$
\lim_{k \to \infty}(u_k ,\alpha_k u_k ,v_k )=(0,w_{\infty},v_{\infty})\ 
\mbox{in}\ C^1 (\overline{\Omega})\times 
C^1 (\overline{\Omega})\times C^1 (\overline{\Omega}_1 )
$$
by passing to a subsequence if necessary, 
and $(w_{\infty},v_{\infty})$ satisfies \eqref{LP2}. 
If $v_{\infty} =0$ in $\overline{\Omega}_1$, then
$$
0=\int_{\Omega}u_k (\lambda -u_k -b\1_{\Omega_1}v_k )\,dx>0
$$
for sufficiently large $k$. But this is impossible. 
Thus $v_{\infty}>0$ in $\overline{\Omega}_1$ by the strong maximum principle. 
Moreover, either $w_{\infty}>0$ in $\overline{\Omega}$ 
or $w_{\infty}=0$ in $\overline{\Omega}$ holds 
by the strong maximum principle again. 
If $w_{\infty}>0$, then Lemma \ref{abthm-l2} yields 
$\mu >0$ and $\lambda <\sigma_1 (b\mu\1_{\Omega_1} )$. 
Assume that $w_{\infty}=0$. 
In this case, the second equation of \eqref{LP2} implies
$$
\Delta v_{\infty}+
v_{\infty}(\mu -v_{\infty}) =0
\ \ \mbox{in}\ \Omega_1 ,\quad 
\partial_n v_{\infty}=0\ \ \mbox{on}\ \partial \Omega_1 .
$$
Hence $v_{\infty}\equiv \mu$ in $\overline{\Omega}_1$ 
and $\mu >0$ must hold because of $v_{\infty}>0$. 
Then by the same argument as in the proof of part (i), 
we can show the existence of a positive function 
$\tilde{u}_{\infty}\in C^1 (\overline{\Omega})$ satisfying
$$
\Delta \tilde{u}_{\infty}+
\tilde{u}_{\infty}( \lambda -b\mu\1_{\Omega_1} ) =0
\ \ \mbox{in}\ \Omega ,\quad 
\partial_n \tilde{u}_{\infty}=0\ \ \mbox{on}\ \partial \Omega .
$$
This means $\lambda =\sigma_1 (b\mu\1_{\Omega_1} )$. 
Therefore, we have completed the proof of part (ii).
\end{proof}
The conclusion of Theorem \ref{abthm} immediately follows 
from Lemma \ref{abthm-l4}.
\section{Proof of Theorem \ref{LP2thm}}
For the local bifurcation branch $\varGamma$ obtained in Lemma, 
let $\varGamma \subset \mathbb{R}\times X$ denote the maximal connected set satisfying
$$
\varGamma \subset 
\{ (\lambda ,w,v)\in (\mathbb{R}\times X)\setminus \{ (\sigma_1 (b\mu\1_{\Omega_1} ) ,0,\mu )\}
:(\lambda ,w,v)\ \mbox{is a solution of \eqref{LP2}}\} .
$$
By an argument essentially parallel to the global bifurcation analysis in Theorem \ref{exthm1},
one can verify that 
$\varGamma \subset \mathbb{R}\times P_{\Omega}\times P_{\Omega_1}$ 
with $P_O =\{ w \in W^{2,p}_n (O):w>0\ \mbox{in}\ \overline{O}\}$
($O=\Omega\ \mbox{or}\ \Omega_{1}$)
is unbounded in $\mathbb{R}\times X$.  
By virtue of Lemmas \ref{abthm-l2} and the elliptic regularity theory,
we know that
the $(\lambda ,v)$ component of $\varGamma$ is uniformly bounded in 
$\mathbb{R}\times W^{2,p} (\Omega_1 )$. 
Hence the $w$ component of $\varGamma \subset \mathbb{R}\times P_{\Omega}\times P_{\Omega_1}$ must blow up at some finete value of
$\lambda$.
In order to show that the blow-up point is $\lambda=0$, we set
\[\lambda_{*}:=\inf\{\lambda\,:\,(\lambda, w, v)\in\varGamma\}.\]
It follows from Lemma \ref{abthm-l2} that
$\lambda_{*}\in  [0,\sigma_{1}(b\mu\1_{\Omega_{1}}))$.
Let $\{ (\lambda_k ,w_k ,v_k )\}_{k=1}^{\infty} \subset \varGamma$ 
be any sequence in $\mathbb{R}\times X$ such that
$\lambda_{k}\searrow\lambda_{*}$.

We will prove that $\{ \min_{\overline{\Omega}}w_k \}_{k=1}^{\infty}$ 
is unbounded by contradiction. Suppose that 
$\{ \min_{\overline{\Omega}}w_k \}_{k=1}^{\infty}$ is bounded. 
\begin{lem}\label{vasylem}
Let $(w,v)$ be any positive solution of \eqref{LP2}.
If $\lambda<b\mu$, then
\[
\dfrac{\lambda}{b}\le v(x)\le\dfrac{C_H}{b}\lambda
\quad\mbox{for any}\ x\in\overline{\Omega}_{1}\]
with some positive constants $C_H>1$ 
independent of $\lambda $.
\end{lem}

\begin{proof}
Integrating the first equation of \eqref{LP2} over $\Omega$, we obtain
\begin{equation}\label{intweq}
b\int_{\Omega_1} v w\,dx \;=\; \lambda\int_{\Omega} w\,dx.
\end{equation}
By virtue of Lemma~\ref{abthm-l2}, we know that the potential function 
\[
K^{(2)}(x):=\frac{w(\lambda - b v)+\mu - v}{1+w}
\]
is continuous and uniformly bounded for $x\in\Omega_1$ and $\lambda>0$.
Then we can use the usual Harnack inequality to find a positive constant
$C_H^{(2)}$, independent of $\lambda$, such that
\begin{equation}\label{Harnack}
\max_{x\in\overline{\Omega}_1} v(x)\;\le\; C_{H}^{(2)}\,
\min_{x\in\overline{\Omega}_1} v(x).
\end{equation}
Combining \eqref{intweq} with \eqref{Harnack}, we observe that
\begin{equation}\label{maxvL1w<lamL1e}
\frac{b}{C_H^{(2)}}\,\max_{x\in\overline{\Omega}_1} v(x)\int_{\Omega_1} w\,dx
\;\le\;
b\,\min_{x\in\overline{\Omega}_1} v(x)\int_{\Omega_1} w\,dx
\;\le\;
b\int_{\Omega_1} v w\,dx
\;=\;
\lambda\int_{\Omega} w\,dx.
\end{equation}
Here we apply the Harnack inequality once more to the first 
equation of \eqref{LP2}:
\[
-\Delta w = \bigl(\lambda - b\1_{\Omega_1}v\bigr)w 
\quad \text{in } \Omega, 
\qquad \partial_{n} w = 0 \quad \text{on } \partial\Omega.
\]
Then there exists a positive constant 
$\widetilde{C}^{(1)}_{H}$, independent of $\lambda$, such that
\[
\max_{x\in \overline{\Omega}} w(x)
\;\le\; \widetilde{C}^{(1)}_{H}
\min_{x\in \overline{\Omega}} w(x).
\]
Consequently, we deduce
\begin{equation}\label{eq:w-integral}
\lambda \int_{\Omega} w\,dx
\;\le\; \lambda \,\widetilde{C}^{(1)}_{H}
\min_{x\in \overline{\Omega}} w(x)\,|\Omega|
\;\le\; \lambda\,\widetilde{C}^{(1)}_{H}\,
\frac{|\Omega|}{|\Omega_{1}|}
\int_{\Omega_{1}} w\,dx.
\end{equation}
Therefore, combining \eqref{maxvL1w<lamL1e} with \eqref{eq:w-integral},
we obtain
\begin{equation}\label{eq:v-upper}
\max_{x\in\overline{\Omega}_1} v(x)
\;\le\; \frac{\widetilde{C}_{H}^{(1)}C_{H}^{(2)}}{b}
\dfrac{|\Omega|}{|\Omega_{1}|}\,\lambda.
\end{equation}
Together with (i) of Lemma~\ref{abthm-l2},
we then obtain the desired estimates with $C_{H}:=\widetilde{C}_{H}^{(1)}C_{H}^{(2)}|\Omega |/|\Omega_1|$.
\end{proof}

\begin{lem}\label{blowupratelem}
Let $(w,v)$ be any positive solution of \eqref{LP2}.
If $\lambda>0$ is sufficiently small, then
\[
\dfrac{C_1}{\lambda}\le w(x)\le\dfrac{C_2}{\lambda}
\quad\mbox{for any}\ x\in\overline{\Omega}\]
with some positive constants $C_1$ and $C_2$ 
independent of $\lambda $.
\end{lem}

\begin{proof}
We have shown that
\eqref{LP2} admits at least one positive solution
for each $\lambda\in (0, \sigma_1 (b\mu\1_{\Omega_1} ,\Omega ))$.
Then we denote any positive solution of \eqref{LP2} for each
$\lambda\in (0, \sigma_1 (b\mu\1_{\Omega_1} ,\Omega ))$ by 
$(w(\lambda), v(\lambda)):=(w(\,\cdot\,,\lambda), v(\,\cdot\,,\lambda)).$ 
In view of Lemma \ref{vasylem},
we note that
\[\widetilde{v}(\lambda ):=\dfrac{v(\lambda )}{\lambda}\] 
has
the uniform upper and lower bounds as follows:
\begin{equation}\label{vtilbdd}
\frac{1}{b}\le \widetilde{v}(\lambda )\le \frac{C_H}{b}\quad \text{in }\Omega_1
\end{equation}
for any $\lambda\in (0,b\mu)$.
Dividing the second equation of \eqref{LP2} by $\lambda^2$,
we can see
\begin{equation}\label{vtil}
-\dfrac{1}{\lambda}\,\Delta \widetilde{v}(\lambda)
= \widetilde{K}(x,\lambda)\,\widetilde{v}(\lambda)
\quad \mbox{in } \Omega_1,\qquad
\partial_{n}\widetilde{v}(\lambda)=0\quad \mbox{on } \partial\Omega_1,
\end{equation}
with
\[
\widetilde{K}(x,\lambda):=
\dfrac{w(\lambda)}{1+w(\lambda )}\,(1-b\widetilde{v}(\lambda))
+\dfrac{\mu}{(1+w(\lambda))\lambda}
-\dfrac{\widetilde{v}(\lambda)}{1+w(\lambda )}.
\]

As the first step of the proof,
we show that $\|w(\lambda )\|_{L^{\infty}(\Omega )}$ blows up as 
$\lambda \to 0^{+}$.
Suppose, for contradiction, that 
$\{\|w_{j}\|_{L^{\infty}(\Omega )}\}$ remains bounded for some
sequence $w_{j}:=w(\lambda_j)$ with $\lambda_j\to 0^{+}$.
Then \eqref{vtilbdd} implies that
the first and third terms of $\widetilde{K}(x,\lambda_j)\widetilde{v}_j$, 
that is,
\[
\frac{w_j}{1+w_j}\,(1-b\widetilde{v}_j)\widetilde{v}_j
\quad \text{and} \quad 
-\frac{\widetilde{v}_j^{\,2}}{1+w_j},
\]
are uniformly bounded in $\overline{\Omega}_1$
with respect to $j$.
On the other hand, 
the second term
\[
\dfrac{\mu\widetilde{v}_j}
{(1+w_j)\lambda_j}
\]
blows up uniformly in $\overline{\Omega}$ as $j\to\infty$.
Consequently, $\widetilde{K}(x,\lambda_j)\widetilde{v}_j$ also blows up uniformly in 
$\overline{\Omega}_1$ as $j\to\infty$.
This, however, is impossible, since integrating \eqref{vtil} yields
\begin{equation}\label{KI}
\int_{\Omega_1}\widetilde{K}(x,\lambda)\,\widetilde{v}(\lambda )\,dx=0
\end{equation}
for any $\lambda$.
Then we obtain
\[
\lim_{\lambda\to 0^+}\|w(\lambda )\|_{L^{\infty}(\Omega )}=\infty
\]
by the contradiction argument.

We next show in more detail that $w(\lambda)$ blows up uniformly in 
$\overline{\Omega}$ at the order of $1/\lambda$. 
To this end, 
it is a standard device to consider the $L^{\infty}$-normalization 
of $w$ defined as follows:
\[
\widehat{w}(\lambda ):=\dfrac{w(\lambda )}{\|w(\lambda )\|_{L^{\infty}(\Omega )}}.
\]
Then dividing the first equation of \eqref{LP2} by
$\|w(\lambda )\|_{L^{\infty }(\Omega )}$,
one can see
\begin{equation}\label{wtile}
-\Delta\widehat{w}(\lambda )=\widehat{w}(\lambda )\,\lambda\,
\bigl(1-b\1_{\Omega_1}\widetilde{v}(\lambda )\bigr)\quad\mbox{in }\Omega,\qquad
\partial_{n}\widehat{w}=0\quad\mbox{on }\partial\Omega.
\end{equation}
By virtue of
$\|\widehat{w}(\lambda )\|_{L^{\infty}(\Omega )}=1$ and \eqref{vtilbdd},
the usual elliptic regularity theorem implies that
$\{\widehat{w}(\lambda )\}$
is uniformly bounded in $W^{2,p}(\Omega )$ for any $p>1$.
By the Sobolev embedding theorem and the strong maximum principle,
one can easily find
a sequence $\{\lambda_j\}$ 
such that $\lambda_{j}\to 0^{+}$
as $j\to\infty$ and
\begin{equation}\label{wtilasy}
\lim_{j\to\infty}\widehat{w}_{j}=1\quad
\mbox{weakly in } W^{2,p}(\Omega )\ \ \mbox{and}\ \ 
\mbox{strongly in } C^{1}(\overline{\Omega }),
\end{equation}
where, for simplicity, we write $\widehat{w}_j:=\widehat{w}(\lambda_{j})$.

Suppose, by contradiction, that either
\begin{equation}\label{ass1}
\limsup_{\lambda\to 0^+}\bigl(\lambda\|w(\lambda)\|_{L^{\infty}(\Omega )}\bigr)=\infty
\end{equation}
or
\begin{equation}\label{ass2}
\liminf_{\lambda\to 0^+}\bigl(\lambda\|w(\lambda)
\|_{L^{\infty}(\Omega )}\bigr)=0.
\end{equation}
We first assume the case \eqref{ass1}.
Then there exists a subsequence of $\{\lambda_{j}\}$
with $\lambda_{j}\to 0^{+}$ as $j\to\infty$ such that
\begin{equation}\label{contass}
\lim_{j\to\infty}\bigl(\lambda_{j}\|w_{j}\|_{L^{\infty}(\Omega )}\bigr)=\infty.
\end{equation}
Taking a subsequence if necessary, we may assume \eqref{wtilasy}.
Dividing the elliptic equation of \eqref{wtile} by 
$\widehat{w}_{j}\lambda_{j}$
and integrating
the resulting expression over $\Omega$, we obtain
\[-\dfrac{1}{\lambda_{j}}
\int_{\Omega}\dfrac{\Delta\widehat{w}_j}{\widehat{w}_j}\,dx=
\int_{\Omega}(1-b\1_{\Omega_1}\tilde{v}_{j})\,dx.\]
For the left-hand side, by integration by parts together with the boundary condition, we obtain
\[
\int_{\Omega}\frac{\Delta \widehat{w}_j}{\widehat{w}_j}\,dx
= \int_{\partial\Omega}\frac{1}{\widehat{w}_j}\,\partial_{n}\widehat{w}_j\,dS
 - \int_{\Omega}\nabla\!\left(\frac{1}{\widehat{w}_j}\right)\cdot\nabla \widehat{w}_j\,dx
= \left\|\frac{\nabla \widehat{w}_j}{\widehat{w}_j}\right\|_{L^{2}(\Omega)}^{2}.
\]
On the other hand, for the right-hand side, by the definition of $b\1_{\Omega_{1}}$ we compute
\[
\int_{\Omega}(1-b\1_{\Omega_1}\widetilde{v}_j)\,dx
= |\Omega_{0}| + \int_{\Omega_1}\bigl(1-b\,\widetilde{v}_j\bigr)\,dx.
\]
Therefore, we obtain
\begin{equation}\label{1-bvj}
\int_{\Omega_{1}}(1-b\widetilde{v}_{j})\,dx=
-|\Omega_{0}|-\dfrac{1}{\lambda_j}
\left\|\frac{\nabla \widehat{w}_j}{\widehat{w}_j}\right\|_{L^{2}(\Omega)}^{2}\le -|\Omega_0|.
\end{equation}
Here we note from \eqref{KI} that
\begin{equation}\label{Int0}
\int_{\Omega_1}K(x,\lambda_j)\,\widetilde{v}_j\,dx=0
\end{equation}
holds for every $j\in\mathbb{N}$. Here we investigate the behavior of 
each term of
\begin{equation}\label{K}
\begin{split}
\widetilde{K}(x, \lambda_{j})\widetilde{v}_j=\ &
\dfrac{\|w_{j}\|_{L^{\infty}(\Omega )}\widehat{w}_j}{
1+\|w_{j}\|_{L^{\infty}(\Omega )}\widehat{w}_j}
(1-b\widetilde{v}_j)\widetilde{v}_j\\
&+
\dfrac{\mu\widetilde{v}_j}{
(1+\|w_{j}\|_{L^{\infty}(\Omega )}\widehat{w}_j)\,\lambda_j}-
\dfrac{\widetilde{v}_j^{\,2}}
{1+\|w_{j}\|_{L^{\infty}(\Omega )}\widehat{w}_j}.
\end{split}
\end{equation}
Owing to the facts that $\|w_j\|_{L^{\infty}(\Omega)}\to\infty$
and
$\widehat{w}_{j}\to 1$ uniformly in $\overline{\Omega}$ by \eqref{wtilasy},
our assumption \eqref{contass} together with \eqref{vtilbdd} 
implies that
\[
\lim_{j\to\infty}
\dfrac{\mu\,\widetilde{v}_j}{
\bigl(1+\|w_{j}\|_{L^{\infty}(\Omega )}\widehat{w}_j\bigr)\,\lambda_j}=0
\quad
\mbox{and}\quad
\lim_{j\to\infty}\biggl(
-
\dfrac{\widetilde{v}_j^{\,2}}
{1+\|w_{j}\|_{L^{\infty}(\Omega )}\widehat{w}_j}\biggr)=0
\]
uniformly
in $\overline{\Omega}_1$.
Concerning the first term,
we remark that \eqref{vtilbdd} leads to
\[
1-b\widetilde{v}_{j}\le 0
\quad\mbox{in }\overline{\Omega}_1 \quad\mbox{for any } 
j\in\mathbb{N}.
\]
In addition, since $\|w_j\|_{L^{\infty}(\Omega)}\to\infty$
and
$\widehat{w}_{j}\to 1$ uniformly in $\overline{\Omega}$,
\eqref{vtilbdd} also yields
\[
\dfrac{\|w_{j}\|_{L^{\infty}(\Omega )}\widehat{w}_j}{
1+\|w_{j}\|_{L^{\infty}(\Omega )}\widehat{w}_j}
\widetilde{v}_j \;\ge\; \dfrac{1}{2b}
\quad\mbox{in }\overline{\Omega}_1 \quad\mbox{for sufficiently large } 
j\in\mathbb{N}.
\]
Therefore, it follows that
\[
\dfrac{\|w_{j}\|_{L^{\infty}(\Omega )}\widehat{w}_j}{
1+\|w_{j}\|_{L^{\infty}(\Omega )}\widehat{w}_j}
(1-b\widetilde{v}_j)\widetilde{v}_j \;\le\;
\dfrac{1}{2b}(1-b\widetilde{v}_j)
\quad\mbox{in } \overline{\Omega}_{1}\quad
\mbox{for sufficiently large } 
j\in\mathbb{N}.
\]
Then, for any small $\varepsilon>0$, there exists a large integer
$j_{0}$ such that if $j\ge j_{0}$, then
\[
\int_{\Omega_1}\widetilde{K}(x,\lambda_{j})\widetilde{v}_j\,dx
\;\le\;\dfrac{1}{2b}\int_{\Omega_1}(1-b\widetilde{v}_j)\,dx+\varepsilon.
\]
Here we recall \eqref{1-bvj} to obtain
\[
\int_{\Omega_1}\widetilde{K}(x,\lambda_{j})\widetilde{v}_j\,dx
\;\le\; -\dfrac{|\Omega_0|}{2b}+\varepsilon<0.
\]
This clearly contradicts \eqref{Int0}.

We next assume the case \eqref{ass2}.
Then there exists a subsequence of $\{\lambda_{j}\}$
with $\lambda_{j}\to 0$ as $j\to\infty$ such that
\begin{equation}\label{contass2}
\lim_{j\to\infty}\bigl(\lambda_{j}\|w_{j}\|_{L^{\infty}(\Omega )}\bigr)=0.
\end{equation}
We may assume \eqref{wtilasy} by passing to a subsequence if necessary.
Hence \eqref{Int0} should hold.
In view of \eqref{K}, one can verify that
the above assumption \eqref{contass2},
together with \eqref{vtilbdd},
ensures that
the first and third terms of $\widetilde{K}(x,\lambda_{j})\widetilde{v}_{j}$
remain uniformly bounded in $\overline{\Omega}_1$ with respect to $j$,
whereas the second term diverges as follows:
\[
\lim_{j\to\infty}
\dfrac{\mu\,\widetilde{v}_j}{
\bigl(1+\|w_{j}\|_{L^{\infty}(\Omega )}\widehat{w}_j\bigr)\,\lambda_j}
=\infty
\quad\mbox{uniformly in } \overline{\Omega}_1.
\]
It then follows that
\[
\lim_{j\to\infty}
\int_{\Omega_1}
K(x,\lambda_j)\,\widetilde{v}_j\,dx=\infty.
\]
This contradicts \eqref{KI}.
Consequently, the above contradiction argument shows that
both \eqref{ass1} and \eqref{ass2} cannot occur.
Therefore, by virtue of the fact that
$\|w(\lambda)\|_{L^{\infty}(\Omega_1)}\to\infty$
as $\lambda\to 0^+$ with \eqref{wtilasy},
we deduce that
$w(\lambda)$ blows up uniformly in $\overline{\Omega}_1$
of order $1/\lambda$ as $\lambda\to 0^+$.
The proof of Lemma \ref{blowupratelem} is complete.
\end{proof}
Having obtained Lemmas \ref{vasylem} and \ref{blowupratelem}, 
which clarify the preliminary properties of $(w,v)$ for small $\lambda>0$, 
we are now in a position to introduce the scaling
\begin{equation}\label{scale}
\widetilde{w}:=\lambda w, 
\qquad 
\widetilde{v}:=\frac{v}{\lambda},
\end{equation}
so as to derive the more detailed asymptotic behavior of positive solutions to \eqref{LP2}.
It is easy to check that $(\widetilde{w}, \widetilde{v})$ satisfies

\begin{equation}
\begin{cases}
\Delta \widetilde{w}+\lambda\,\widetilde{w}\,
(\,1 -b\1_{\Omega_1}\,\widetilde{v}\,)=0\ \ &\mbox{in}\ \Omega,\\[2pt]
\Delta \widetilde{v}+\dfrac{\lambda\,\widetilde{v}}{\lambda+\widetilde{w}}
\left\{ \,\widetilde{w}\,(\,1 -b\widetilde{v}\,)+\mu -\lambda\widetilde{v}
\,\right\} =0
\ \ &\mbox{in}\ \Omega_1 ,\\[2pt]
\partial_n \widetilde{w}=0\ \
\ \ &\mbox{on}\ \partial \Omega,\\
\partial_n \widetilde{v}=0\ \
\ \ &\mbox{on}\ \partial \Omega_1.
\end{cases}
\nonumber
\end{equation}
We aim to parametrize the branch of positive solutions to the above system
for small $\lambda>0$ by means of the Lyapunov-Schmidt reduction 
and to describe their asymptotic behavior as $\lambda\to 0^{+}$.
For this purpose we decompose the unknowns $(\widetilde{w},\widetilde{v})$ as
\begin{equation}\label{LP22}
\begin{split}
&\widetilde{w}=s+\lambda\phi,\qquad 
s=\frac{1}{|\Omega|}\int_{\Omega}\widetilde{w}\,dx,\\
&\widetilde{v}=t+\lambda\psi,\qquad 
t=\frac{1}{|\Omega_{1}|}\int_{\Omega_{1}}\widetilde{v}\,dx,
\end{split}
\end{equation}
where 
\[
\phi\in W^{2,p}_{n,\#}(\Omega),\qquad 
\psi\in W^{2,p}_{n,\#}(\Omega_{1}).
\]
Here $W^{2,p}_{n,\#}(O)$ denotes the subspace of 
$W^{2,p}_{n}(O)$ consisting of functions with vanishing mean value.
Thus the unknowns are 
\[
(s,t,\phi,\psi)\in \mathbb{R}^{2}
\times W^{2,p}_{n,\#}(\Omega)\times W^{2,p}_{n,\#}(\Omega_{1}),
\]
and $(\phi,\psi)$ satisfy the system
\begin{equation}\label{LS-comp}
\begin{cases}
\Delta \phi+(s+\lambda\phi)\bigl(1-b\1_{\Omega_1}(t+\lambda\psi)\bigr)=0
& \text{in }\Omega,\\[2pt]
\Delta \psi+\dfrac{t+\lambda\psi}{\lambda+s+\lambda\phi}
\Bigl\{(s+\lambda\phi)(1-b(t+\lambda\psi))+\mu-\lambda(t+\lambda\psi)\Bigr\}=0
& \text{in }\Omega_{1},\\[2pt]
\partial_{n}\phi=0 & \text{on }\partial\Omega,\\[2pt]
\partial_{n}\psi=0 & \text{on }\partial\Omega_{1}.
\end{cases}
\end{equation}

Introducing the projections 
$Q:L^{p}(\Omega)\to L^{p}_{\#}(\Omega)$ and 
$Q_{1}:L^{p}(\Omega_{1})\to L^{p}_{\#}(\Omega_{1})$, 
we can rewrite \eqref{LS-comp} as a coupled system for the 
$\mathbb{R}$-components and the mean-zero components. 
By applying the inverse Neumann Laplacians 
$(-\Delta)^{-1}_{\Omega,\#}$ and $(-\Delta)^{-1}_{\Omega_{1},\#}$, 
this system is equivalently formulated as the nonlinear operator equation
\[
F(\lambda,s,t,\phi,\psi)=0,
\]
with components
\[
\begin{aligned}
&F_{1}(\lambda,s,t,\phi,\psi)
=\int_{\Omega}(s+\lambda\phi)\bigl(1-b\1_{\Omega_1}(t+\lambda\psi)\bigr)\,dx,\\[4pt]
&F_{2}(\lambda,s,t,\phi,\psi)
=\int_{\Omega_{1}}
\frac{t+\lambda\psi}{\lambda+s+\lambda\phi}
\Bigl\{(s+\lambda\phi)(1-b(t+\lambda\psi))+\mu-\lambda(t+\lambda\psi)\Bigr\}\,dx,\\[4pt]
&F_{3}(\lambda,s,t,\phi,\psi)
=\phi-(-\Delta)^{-1}_{\Omega,\#}Q\bigl[(s+\lambda\phi)(1-b\1_{\Omega_1}(t+\lambda\psi))\bigr],\\[4pt]
&F_{4}(\lambda,s,t,\phi,\psi)
=\\
&\psi-(-\Delta)^{-1}_{\Omega_{1},\#}Q_{1}\!\left[
\frac{t+\lambda\psi}{\lambda+s+\lambda\phi}
\Bigl\{(s+\lambda\phi)(1-b(t+\lambda\psi))+\mu-\lambda(t+\lambda\psi)\Bigr\}
\right].
\end{aligned}
\]
Setting $\lambda=0$, we obtain the system
\[
F(0,s,t,\phi,\psi)=
\begin{pmatrix}
s\displaystyle\int_{\Omega}(1-b\1_{\Omega_1}t)\,dx \\[10pt]
\dfrac{t}{s}\displaystyle\int_{\Omega_{1}}(s(1-bt)+\mu)\,dx \\[10pt]
\phi-(-\Delta)^{-1}_{\Omega,\#}Q[\,s(1-b\1_{\Omega_1}t)\,]\\[6pt]
\psi-(-\Delta)^{-1}_{\Omega_{1},\#}Q_{1}\!\left[\dfrac{t}{s}(s(1-bt)+\mu)\right]
\end{pmatrix}.
\]
In order to find a solution $(s_{0}, t_{0}, \phi_0, \psi_0)$ with
$s_{0}$,
$t_{0}>0$ to the limit equation $F(0,s,t,\phi,\psi)=0$.
we remark that the first two components of the equation 
are reduced to the algebraic equations
\[
|\Omega|-b\,t|\Omega_1|=0,\qquad
s(1-bt)+\mu=0.\]
Then $(s_{0},t_{0})$ is uniquely determined by
\[
t_{0}=\frac{|\Omega|}{b|\Omega_{1}|},\qquad
s_{0}=\frac{\mu|\Omega_{1}|}{|\Omega_{0}|},
\]
and the associated pair $(\phi_{0},\psi_{0})$ is given by
\[
\phi_{0}
=\mu\,(-\Delta)^{-1}_{\Omega,\#}\!\left[
\frac{|\Omega_{1}|}{|\Omega_{0}|}\,\1_{\Omega_{0}}-\1_{\Omega_{1}}
\right],
\qquad
\psi_{0}=0.
\]
Thus $(s_{0},t_{0},\phi_{0},\psi_{0})$ is the explicitly expressed solution of $F(0,s,t,\phi,\psi)=0$, which serves as the base point for the reduction.

Next we compute the Fr\'echet derivative of $F$ with respect to $(s,t,\phi,\psi)$ at the base point. 
For the computation of the Fr\'echet derivative, it is convenient to introduce 
the simplified notation
\[
F_{j}^{0}(s,t):=F_{j}(0,s,t,\phi,\psi),\qquad j=1, 2, 3, 4,
\]
which denote the first two components of $F$ at $\lambda=0$ 
(they do not depend on $(\phi,\psi)$ in this case).  
Explicitly, they take the form
\[
F_{1}^{0}(s,t)=s\,(|\Omega|-b\,t|\Omega_{1}|),\qquad
F_{2}^{0}(s,t)=|\Omega_{1}|\,t\Bigl(1-bt+\frac{\mu}{s}\Bigr).
\]
Differentiating and evaluating at $(s_{0},t_{0})$ yields
\[
\partial_{s}F_{1}^{0}(s_{0},t_{0})=0,\qquad
\partial_{t}F_{1}^{0}(s_{0},t_{0})=-s_{0}b|\Omega_{1}|,
\]
\[
\partial_{s}F_{2}^{0}(s_{0},t_{0})=-\frac{t_{0}}{s_{0}}|\Omega_{0}|,
\qquad
\partial_{t}F_{2}^{0}(s_{0},t_{0})=-b\,t_{0}|\Omega_{1}|.
\]
Thus the Jacobian matrix of the mapping 
\[
(s,t)\longmapsto \bigl(F_{1}^{0}(s,t),\,F_{2}^{0}(s,t)\bigr)
\]
evaluated at the base point $(s_{0},t_{0})$ is given by
\[
J=
\begin{pmatrix}
\partial_{s}F_{1}^{0}(s_{0},t_{0}) & \partial_{t}F_{1}^{0}(s_{0},t_{0}) \\[6pt]
\partial_{s}F_{2}^{0}(s_{0},t_{0}) & \partial_{t}F_{2}^{0}(s_{0},t_{0})
\end{pmatrix}
=
\begin{pmatrix}
0 & -\,s_{0}b\,|\Omega_{1}| \\[6pt]
-\dfrac{t_{0}}{s_{0}}\,|\Omega_{0}| & -\,b\,t_{0}\,|\Omega_{1}|
\end{pmatrix}.
\]
In particular, its determinant is
\[
\det J
=\bigl(\partial_{s}F_{1}^{0}\,\partial_{t}F_{2}^{0}
- \partial_{t}F_{1}^{0}\,\partial_{s}F_{2}^{0}\bigr)(s_{0},t_{0})
=-\,b\,t_{0}\,|\Omega_{0}|\,|\Omega_{1}|<0.
\]
Hence $J$ is invertible.

For the infinite-dimensional components, the derivative is given by
\begin{align*}
&  \partial_{s}F_{3}^{0}(s_{0}, t_{0})=-(-\Delta)^{-1}_{\Omega,\#}Q[\,1-b\1_{\Omega_1}t_{0}\,],\\
&\partial_{t}F_{3}^{0}(s_{0}, t_{0})=s_{0}(-\Delta)^{-1}_{\Omega,\#}Q[\,b\1_{\Omega_1}\,],\\
&\partial_{\phi}F_{3}^{0}(s_{0}, t_{0})=\mathrm{Id}_{W^{2,p}_{n,\#}(\Omega)},
\end{align*}
and
\[
\partial_{s}F_{4}^{0}(s_{0}, t_{0})=0,\qquad
\partial_{t}F_{4}^{0}(s_{0}, t_{0})=0,\qquad
\partial_{\psi}F_{4}^{0}(s_{0}, t_{0})=\mathrm{Id}_{W^{2,p}_{n,\#}(\Omega_{1})}.
\]
Hence the Fr\'echet derivative at the base point can be expressed in block form as
\[
D_{(s,t,\phi,\psi)}F(0,s_{0},t_{0},\phi_{0},\psi_{0})
=
\begin{pmatrix}
J & 0 & 0 \\[6pt]
\mathcal{T}_{1} & \mathrm{Id}_{W^{2,p}_{n,\#}(\Omega)} & 0 \\[6pt]
\mathcal{T}_{2} & 0 & \mathrm{Id}_{W^{2,p}_{n,\#}(\Omega_{1})}
\end{pmatrix},
\]
where
\[
\mathcal{T}_{1}=\bigl(-(-\Delta)^{-1}_{\Omega,\#}Q[\,1-b\1_{\Omega_1}t_{0}\,],\;
s_{0}(-\Delta)^{-1}_{\Omega,\#}Q[\,b\1_{\Omega_1}\,]\bigr),
\qquad
\mathcal{T}_{2}=(0,0).
\]
Since $J$ is invertible and both diagonal operators 
$\mathrm{Id}_{W^{2,p}_{n,\#}(\Omega)}$ and 
$\mathrm{Id}_{W^{2,p}_{n,\#}(\Omega_{1})}$ are identities, 
the whole operator is invertible. 
This establishes the nondegeneracy of the base point and allows us to apply the implicit function theorem.
Hence the implicit function theorem gives a neighborhood 
$\widetilde{\mathcal{U}}$
of 
\[(0,s_{0}, t_{0}, \phi_{0}, 0)
\in\mathbb{R}^{3}\times W^{2,p}_{n,\#}(\Omega)
\times W^{2,p}_{n,\#}(\Omega_1)\]
and a small $\kappa>0$ such that
\begin{align*}
&\{\,(\lambda, s, t, \phi, \psi)\in\widetilde{\mathcal{U}}\,:\,
F(\lambda, s, t, \phi, \psi)=0\,\}\\
=&\{\, (\lambda, s(\lambda ), t(\lambda ), \phi(\lambda ), \psi(\lambda ))\,:\,\lambda\in (-\delta, \delta )\,\}\,(\,=:\widetilde{\varGamma}_{0}),
\end{align*}
with some continuously differentiable functions
\[
(s(\lambda ), t(\lambda ), \phi(\lambda ), \psi(\lambda ))
\in (0,\infty)\times (0,\infty)\times W^{2,p}_{n,\#}(\Omega )\times
W^{2,p}_{n,\#}(\Omega_1)
\]
for $\lambda\in (-\kappa, \kappa)$,
satisfying
$
(s(0), t(0), \phi(0), \psi(0))=(s_0, t_0, \phi_0, 0).
$
By \eqref{scale}, we construct a curve of positive solutions of \eqref{scale}. 
To this end, we introduce the neighborhood
\[
\mathcal{U}:=\biggl\{
(\lambda, w, v)\,:\,
(w, v)=\biggl(
\dfrac{s}{\lambda}+\phi,\, \lambda (t+\lambda\psi)
\biggr),\
(\lambda, s, t, \phi, \psi)\in\widetilde{\mathcal{U}},\ \lambda>0\biggr\}
\]
by means of the change of variables \eqref{scale} and \eqref{LP22}.
Then, applying the same change of variables to 
$\widetilde{\varGamma}_{0}|_{0<\lambda<\kappa}$, 
we see that all positive solutions of \eqref{LP2} contained in
$\mathcal{U}$ form an unbounded curve parameterized by $0<\lambda<\kappa$ as
\begin{align*}
\varGamma_{0+}=
\biggl\{\,(\lambda, w(\lambda ), v(\lambda))\in (0,\kappa)\times X\,:\, 
&w(\lambda )=\dfrac{s(\lambda )}{\lambda}+\phi (\lambda ),\\ 
&v(\lambda )=\lambda \bigl(t(\lambda )+\lambda \psi (\lambda )\bigr)\,\biggr\}.
\end{align*}
The proof of Theorem \ref{LP2thm} is complete.

\section*{Declarations}

\subsection*{Conflict of interest}
The authors declare that they have no competing interests.

\subsection*{Ethical approval}
This article does not contain any studies with human participants or animals performed by any of the authors.

\subsection*{Informed consent}
Not applicable.

\subsection*{Data availability}
Data sharing is not applicable to this article as no datasets were generated or analyzed during the current study.


\bibliographystyle{spmpsci}
\bibliography{KutoOeda}

\end{document}